\documentclass{amsart}
\usepackage[english]{babel}
\usepackage{amsfonts}
\usepackage{amssymb}
\usepackage{amsthm}
\usepackage{amsmath}
\usepackage{amscd}
\usepackage{newlfont}

\newtheorem{Theorem}{Theorem}
\theoremstyle{remark}

\newtheorem{Proposition}[Theorem]{Proposition}

\def\bysame{\leavevmode\hbox to3em{\hrulefill}\thinspace}
\newcommand{\vettore}[1]{\overset{\rightarrow}{#1}}

\title{Loops of Lie type: a $p$-adic example}

\author{Raffaello Caserta}

\begin{document}

\maketitle\label{chapterCaserta}

\section*{Introduction}

The purpose of this survey is to give neither a full detailed course of differential geometry and Lie theory nor of differential loops, we would rather sketch a path which leads us from the geometry of curved surfaces to the non associative features of the local algebraic structures associated to them.

The sum of two vectors on the Euclidean plane according to the so called parallelogram rule is an elementary geometrical process. Assuming that three points $\Omega, A, B$ of the Euclidean plane are given, the sum of the vectors $\vettore{\Omega A}, \vettore{\Omega B}$ is the vector $\vettore{\Omega C}$, where $C$ is the resulting point of the concatenation of the vectors $\vettore{\Omega A}$ and $\vettore{v}$, with $\vettore{v}$ being the unique parallel vector to $\vettore{\Omega A}$ applied in $B$. This sum between vectors also defines a commutative operation $A+B=C$ on the set of the points of the plane. The parallelogram rule takes advantage on the identification between applied vectors and segments, on the parallelism of lines and on the uniqueness of a parallel line to a fixed one and through a point off of it. Each of these are characteristic of the flat Euclidean plane, and it makes sense to investigate how to generalize each of these to a non-flat surface, if possible at all, which is exactly the starting point of the work of M. Kikkawa about differential loop defined on a differentiable real manifold with an affine connection. In this chapter we will offer an example based on the work of Kikkawa and developed on a $p$-adic manifold.

In order to do that, we will firstly introduce some basic notions of differential geometry such as differential manifolds, germs of functions, derivations, tangent space in a point and global tangent space (the tangent bundle), the differential and the vector fields. Such topics are quite familiar for a real algebraic subspace or manifolds and are part of the course of real function analysis. What we will recall below is important in order to generalize familiar geometrical objects to the contexts where the spatial representation fails, due to the complexity of Euclidean surfaces, or even for the exoticity of non Euclidean objects which can be constructed by changing the supporting field. 

What we are going to introduce holds both for real and $p$-adic (analytic) manifolds. Afterwards we will give account of the work of Kikkawa with local loops, which are defined by means of affine connection, geodesics, parallelism and the exponential map, and in the last section we will construct a loop on the surface of the $p$-adic sphere. Many problems have to be solved if one tries to export the differential geometry of the real sphere, where a local loop is always defined in a neighbourhood of one pole,  to the $p$-adic sphere: the most evident is due to the fact that the differential structure of the real sphere is induced by the Riemannian metric, which has no counterpart on a $p$-adic manifold. Even the customary affine connection of a real sphere can be easier induced by the Riemannian metric and a geometrical meaningfully analogue in the $p$-adic geometry is far to be reached. In order to compute the geodesic curves of the $p$-adic $2$-sphere, we need a covariant derivative which is defined by parallel displacing the vectors of the tangent bundle of the manifold, and the notion of parallelism which is induced by the connection. We will see that for the $p$-adic $2$-sphere as a homogeneous space, hence, as principal bundle is a reductive space, therefore, a connection can be defined by differentially splitting at each point the tangent space, a Lie algebra, into the direct sum of a vertical and a horizontal component.

\section{Elements of differential geometry}

In the first half of the XIX century, K. F. Gau{\ss} proposed in his work \textit{Disquisitiones generales circa superficies curvas} (1827) the possibility to consider a real (hyper) surface as a space itself. A real manifold, as a subspace of $\mathbb{R}^n$, already inherits many geometrical structures, the tangent space for instance, which do not necessarily need to be defined intrinsically within the variety, they are immediately available since they are defined globally in the entire space $\mathbb{R}^n$. However, intrinsic definitions and the development of an abstract theory of manifolds from inside a manifold have a double benefit: on the one hand it makes possible to study ``strange'' manifolds, real or not, on the other hand we can understand non-Euclidean or even not real manifolds which cannot be embedded into a real Euclidean space. Thanks to this approach, we can use the same theoretical frame developed for a real manifold in order to provide an example which is built on a $p$-adic manifold.

\subsection{Differential manifold}

A real differential manifold is a geometrical object which locally behaves like a portion of the Euclidean real space, therefore, locally, it is equipped with a system of coordinates as well as an Euclidean space is equipped with a global system of coordinates. Already Euler employed  parameters in order to describe a surface, that is, by varying the values of the parameters we get all the points of the manifold.

\subsubsection{Manifold and local coordinates}

Let $U,V$ be open subsets of a finite dimensional real vector space. A map $f:U\to V$ is of class $\mathcal{C}^\infty(U)$ or \textit{smooth} if there exists each partial derivative of $f$ of every order. Throughout this work we will always deal with analytic manifolds as differentiable manifolds\footnote{By means of Lie theory we will introduce also a $p$-adic analytic manifold. Although historically many tools of differential geometry have been developed in an Archimedean metrical context, many propositions still hold when the supporting field is any field of characteristic zero, in particular they hold for the field of $p$-adic numbers.}.\\

Let $M$ be a topological space. An $n$-dimensional \textit{smooth atlas} of $M$ is a set of triples $\{(\varphi_i,U_i,A_i)\}_{i\in I}$ with the \textit{charts} $\varphi_i:U_i\to A_i$, $i\in I$, such that 
\begin{enumerate}
\item $\{U_i\}_{i\in I}$ is an open covering of $M$,
\item $\{A_i\}_{i\in I}$  a set of open subsets of $\mathbb{R}^n$,
\item $\varphi_i$, $i\in I$, are all homeomorphisms and 
\item $\varphi_i\circ\varphi_j^{-1}:\varphi_j(U_i\cap U_j)\to \varphi_i(U_i\cap U_j)$,$i,j\in I$, are smooth diffeomorphisms. 
\end{enumerate}

A topological space $M$ is a \textit{real differential manifold} of dimension $n\in\mathbb{N}$ if it is equipped with an $n$-dimensional \textit{atlas}\footnote{A more rigorous definition requires that the space $M$ is also a local compact Hausdorff space satisfying the second numerability axiom (hence a paracompact manifold), moreover one should define classes of atlas which are equivalent up to diffeomorphism. It is also worth to emphasize that the dimension of the manifold is independent by the chosen atlas}.\\

As we already pointed out, unlike the Euclidean spaces, that a manifold does not have a global system of coordinates, that is, each chart $(U,\varphi)$ carries a system of coordinates called \textit{local coordinates}, which we will denote by $\varphi=(x_1,x_2,\ldots, x_n)$, that is, $\varphi(m)=(\varphi_1(m),\varphi_2(m),\ldots,\varphi_n(m))=(x_1,x_2,\ldots,x_n)\in\mathbb{R}^n$. One can transform the coordinates of a point in a system to another system by means of the Jacobian matrix of the transition map. It is worth to emphasize that each time a new object is defined in terms of theoretical properties, there always exists a representation in terms of local coordinates in the same way as an abstract vector has a numerical representation with respect to a basis.\\

As an example of smooth manifolds we have the analytic manifolds.
Let $K$ be a field of null characteristic either Archimedean (f.i. $K=\mathbb{R}$ or $\mathbb{C}$) or a local (f.i. $K=Q_p$) and let $V,W$ be two normed $K$-vector space, $U$ an open set of $V$ and $f:U\subseteq V\to W$ a map. Say that $f$ is \textit{strictly differentiable}\footnote{The strictly differentiation guarantees the local existence of a solution of a first order differential equation, whence it allows us to prove the equivalence between curves and vector fields, that is. Nevertheless, once the differential structure is encoded into the Lie group of transformation, the algebraic structure simplifies the approach.} in $u_0\in U$ if there exists a continuous linear map $D_{u_0}f:V\to W$ such that for all $\varepsilon>0$ there exists an open neighborhood $U_\varepsilon\subseteq U$ of $u_0$ such that
$$
\|f(\underline{v}_1)-f(\underline{v}_2)-D_{u_0}f(\underline{v}_1-\underline{v}_2)\|\leq\varepsilon\|\underline{v}_1-\underline{v}_2\|
$$
for all $\underline{v}_1,\underline{v}_2\in U_\varepsilon$. In particular, a map $f:U\subseteq V\to W$ is \textit{locally analytic} if for all $x_0\in U$ there exists a ball $B_\varepsilon(x_0)\subseteq U$ around $x_0$ and a power series $F$ such that $f(x)=F(x-x_0)$ for all $x\in B_\varepsilon(x_0)$ (see \cite{schneider2011p} page  38). If $f:U\to V$ is locally analytic, then $f$ is strictly differentiable in each $x_0\in U$ and $x\mapsto\mathrm{d}_x f$ is locally analytic (see \cite{schneider2011p} Proposition 6.1 page  39). In particular $f$ is locally constant on $U$ if and only if $\mathrm{d}_x f=0$ for all $x\in U$ (see \cite{schneider2011p} Remark 6.2 page  40). Of course, compositions, Cartesian products, and projections of locally analytic maps are also analytic.
As usual, $K^n$ is a locally analytic manifold over $K$ and the matrix group $\mathrm{GL}_n(K)$ as a submanifold of $K^{n^2}$ is a group and a differential manifold (a Lie group) (Section \ref{casertalie-group} p. \pageref{casertalie-group}) with respect to the matrix product (see Examples page 90 in \cite{schneider2011p}).

\subsubsection{Tangent space to a manifold in a point}
In this section we will construct three isomorphic vector spaces in order to give the definition of tangent space in a point to a manifold $M$ by means of derivations, of the algebra of germs and of smooth curves. We will also later define the applied vectors space on the manifold $M$ by introducing the tangent bundle, two applied vectors being parallel according to the so called connection or, equivalently, the covariant derivative.

\textbf{Class of curves}
The most familiar way to define a tangent vector in $m$ to $M$ is to take a smooth curve through the fixed point $m$ and to consider the tangent vector in $m$ to the curve. Let $m\in M$,  $\gamma:(-\varepsilon,\varepsilon)\to M$ be a smooth curve centered in $m$, that is, $\gamma(0)=m$, and let $(\varphi,U)$ be a chart with  $(-\varepsilon,\varepsilon)\subseteq U$ through $m=\gamma(0)$. Set $\gamma_\varphi:=f\circ\gamma$, then  
\begin{equation}
\dot{\gamma}_\varphi(0):=\mathrm{d}/\mathrm{d}t|_{t=0}(\varphi\circ\gamma)\in\mathbb{R}^n
\end{equation}
is a $n$-tuple of scalars to be understood as the local coordinates of ``a tangent vector'' in $m$ to $\gamma$. If $(\psi,V)$ is another chart through $m$ and $J=(\partial_i(\psi\circ\varphi^{-1})/\partial x_j)$ is the Jacobian matrix of the map $\psi\circ\varphi^{-1}$ in $\varphi(m)$, then 
$$
\dot{\gamma}_\varphi(0)=J\,\dot{\gamma}_\psi(0),
$$
that is, they are tensorial. Whence, given two curves $\gamma,\chi$, the relation $\chi\sim\gamma$ if and only if $\gamma(0)=\chi(0)$ and $\dot{\chi}_\varphi(0)=\dot{\gamma}_\varphi(0)$ for some chart $\varphi, \psi$ is an equivalence whose classes will be denoted by $[\gamma]$ or $\dot{\gamma}(0)$, due to the independence of the charts. Each class ideally identifies a vector of $\mathbb{R}^n$ applied in $m$ up to local coordinates representation. We can define two operations
$$
[\gamma_1]+[\gamma_2]:=[\gamma_1+\gamma_2],\quad \lambda[\gamma]:=[\lambda\gamma],\quad [\gamma_1],[\gamma_2],[\gamma]\in\Gamma_m,\, \lambda\in\mathbb{R}
$$
by setting for all $t\in(-\varepsilon,\varepsilon)$
\begin{eqnarray*}
\gamma_1+\gamma_2 &:& t\mapsto \varphi^{-1}(\varphi(\gamma_1(t))+\varphi(\gamma_2(t)))\\ \lambda\gamma &:&t\mapsto\varphi^{-1}(\lambda\varphi(\gamma(t))).
\end{eqnarray*}
Both operations are well defined since
$$
(\gamma_1+\gamma_2)'_\varphi(0)=(\dot{\gamma}_1)_\varphi(0)+(\dot{\gamma}_2)_\varphi(0),\quad (\lambda\gamma)'_\varphi=\lambda\dot{\gamma}_\varphi(0).
$$
Hence, with respect to these operations, $\Gamma_m$ is a real vector space of dimension $n$, a basis of which is given for example by 
$[\varphi^{-1}(t\underline{e}_i+\varphi(m))]$, $i=1,2,\ldots, n$, where
$\{\underline{e}_i:i=1,2,\ldots, n\}$ is the canonical basis of $\mathbb{R}^n$ and $t\in \mathbb{R}$ such that $t\underline{e}_i+\varphi(m)\in U$, for all $i=1,2,\ldots n$.

\textbf{Derivation in a point}
If $f:M\to\mathbb{R}$ is a smooth map and $\gamma$ is a smooth curve centered in $m$, then the scalar $\mathrm{d}/{\mathrm{d}\tau}|_{\tau=0}(f\circ\gamma)$
does not depend either from the chart or from the curve, since, in local coordinates $\varphi=(x_i)_i$, we have
$$
\mathrm{d}/{\mathrm{d} t}|_{t=0}(f\circ\gamma)=
{\mathrm{d}/{\mathrm{d}t}|_{t=0}(f\circ\varphi^{-1}\circ\varphi\circ\gamma)}=\sum_{i=1}^n\partial x_i(f\circ\varphi^{-1})(\varphi(m))\dot{\gamma}_\varphi(0)_i.
$$

Define a \textit{derivation} $D_{|m}$ of $M$ in $m\in M$ as  a linear form
$$
D_{|m}:\mathcal{C}_m^\infty(M)\to\mathbb{R}
$$
which satisfies the Leibniz law $D_{|m}(fg)=f(m)D_{|m}(g)+g(m)D_{|m}(f)$. If $D_{|m},D^1_{|m},D^2_{|m}$ are derivations in $m$, then the operations
$$
(D^1_{|m}+D^2_{|m})(f):=D^1_{|m}(f)+D^2_{|m}(f),\quad (\lambda D^1_{|m})(f):=\lambda D_{|m}(f),\, \lambda\in\mathbb{R},
$$
turn the set $\mathrm{Der}_m(M)$ of all derivations in a vector space of the same dimension of $M$, as the following proposition shows.

\begin{Proposition}
A basis of $\mathrm{Der}_m(M)$ in local coordinates $\varphi=(x_i)_i$ is the set $\{\partial x_i|_m\}_{i=1,2,\ldots,n}$, where
$$
\partial x_i|_m(f):=\partial x_i(f\circ\varphi^{-1})(\varphi(m)), \, i=1,2,\ldots, n.
$$
Therefore, with respect the previous basis, the derivation $D_{|m}$ has components $\{D|_m(\varphi_i)\}$, that is,
$$
D|_m=\sum_{i=1}^n D|_m(\varphi_i) \partial x_i|_m.
$$
Moreover, if $\psi=(y_i)_i$ is another chart and $J$ the Jacobian of $\psi^{-1}\circ\varphi$ in $\varphi(m)$, then $D|_m(\psi)=JD|_m(\varphi)$.
\end{Proposition}

\begin{proof}
Let $\mathbf{1}:m\in M\mapsto 1\in\mathbb{R}$ be the unitary constant function. Then:
\begin{enumerate}
\item $D|_m(\mathbf{1})=D|_m(\mathbf{1}\mathbf{1})=2D|_m(\mathbf{1})$, hence $D|_m(\mathbf{1})=0$ and
\item if $f(m)=g(m)=0$, then $D|_m(fg)=0$ for $f,g\in \mathcal{C}^\infty(M)$.
\end{enumerate}
Let $m_0\in M$ be fixed, $f\in\mathcal{C}^\infty(M)$ and $\varphi$ a chart through $m_0$. Set $g:=f-f(m_0)\mathbf{1}$ and $h:=\varphi-\varphi(m_0)\mathbf{1}$ (which is still a chart through $m_0$). One has therefore $g(m_0)=0$ and $h(m_0)=0$, hence
$$
g=\sum_i\partial x_i(g\circ h^{-1})(0)h_i+\sum_{ij}a_{ij}h_ih_j
$$
for some smooth maps $a_{ij}$, whence $D|_m(g)=\sum_i\partial x_i(g\circ  h^{-1})(0)D|_m(h_i)$. 
Since $h=\varphi-\varphi(m_0)\bold{1}$, one has $x=h(m)=\varphi(m)-\varphi(m_0)$, that is, $h^{-1}(x)=m=\varphi^{-1}(x+\varphi(m_0))$, therefore
$$
(g\circ h^{-1})(x)=((f-f(m_0)\bold{1})\circ h^{-1})(x)=(f\circ h^{-1}-f(m_0))(x)=
(f\circ\varphi^{-1})(x+\varphi(m_0))
$$
which implies
$$
\partial x_i (g\circ h^{-1})(h(m))=\partial x_i (f\circ \varphi^{-1})(\varphi(m)).
$$
Moreover, $D|_m(h_i)=D|_m(\varphi_i-\varphi(m_-0)\bold{1})=D|_m(\varphi)$
and $D|_m(g)=D|_m(f-f(m_0)\bold{1})=D|_m(f)$, thus, in local coordinates
$$
D|_m(f)=\sum_i\partial x_i(f\circ \varphi^{-1})(\varphi(m))D|_m(\varphi_i).
$$
Finally, given another chart $(\psi_i)=(y_i)$,
\begin{multline*}
\partial y_i (f)=
\partial \psi_i (f\circ\psi^{-1})=
\partial \psi (f\circ\varphi^{-1}\circ\varphi\circ\psi^{-1})=\\
\sum_k \partial x_k (f\circ\varphi^{-1}) \partial x_i (\varphi\circ\psi^{-1})_k=
\sum_k \partial x_k(f) \partial x_i (\varphi\circ\psi^{-1})_k
\end{multline*}.
\end{proof}

\begin{Proposition}
The vector spaces $\Gamma_m$ and $\mathrm{Der}_m(M)$ are isomorphic.
\end{Proposition}

\begin{proof} 
Let $\gamma:I\to M$ be a smooth curve. The map 
\begin{equation}\label{casertaderivation-induced-by-a-curve}
D_{\gamma,t}:f\in\mathcal{C}^\infty(M)\mapsto D_{\gamma,t}(f):=\mathrm{d}/\mathrm{d}\tau|_{\tau=t}(f\circ\gamma)
\end{equation}
is a derivation in $\gamma(t)$ for all $t\in I$, in particular $D_{\gamma,t}=\sum\dot{\gamma}_\varphi(t) \partial\varphi_i|_{\gamma(t)}$ for some chart $(U,\varphi)$. Conversely, if $D|_m$ is a derivation, for a chart $\varphi=(\varphi_i)$, let 
$$
\gamma(\tau):=(D|_m(\varphi_1),\ldots,D|_m(\varphi_n))(\tau-t)+\varphi(m)
$$
one has $\gamma(t)=\varphi(m)$ and for $\tilde{\gamma}=\varphi^{-1}\circ\gamma$ one has
\begin{multline*}
\mathrm{d}/{\mathrm{d}\tau}|_{\tau=t}(f\circ\tilde{\gamma})=
\mathrm{d}/{\mathrm{d}\tau}|_{\tau=t}(f\circ\varphi^{-1}\circ\gamma)=\sum_i\partial x_i(f\circ\varphi^{-1})(\varphi(m))\dot{\gamma}_i(t)=\\
=\sum_i\partial x_i(f\circ\varphi^{-1})(\varphi(m))D|_m(\varphi_i)=D|_m(f).
\end{multline*}
\end{proof}

\textbf{Algebra of germs and cotangent space}
Let $m\in M$ be fixed and denote by $C_m^\infty(M)$ the set of all pairs $(U,f)$ with $U$ open, $m\in U$ and $f:U\to\mathbb{R}$ smooth. With respect to the sum and product of functions, $C_m^\infty(M)$ is a ring and the evaluation
$$
\mathrm{eval}_m:(U,f)\in\mathcal{C}_m^\infty\mapsto f(m)\in\mathbb{R}^n
$$
is a ring epimorphism. The kernel $I$ is an ideal which leads to the equivalence $(U,f)\sim (V,g)$ if and only if $f(x)=g(x)$ for all $x\in U\cap V$. We term the class $[(U,f)]$ \textit{germ}, that is, the class of all functions which take the same values in some neighborhood of $m$, and we denote by  $\mathfrak{C}_m^\infty(M)$ the factor ring $\mathfrak{C}^\infty_{m}(M):=C_m^\infty(M)/I=C_m^\infty(M)/\sim$. The ring $\mathfrak{C}^\infty_{m}(M)$ is graded and local and $\mathfrak{m}_{m}$ is the unique maximal ideal, the ideal of all germs which are null in $m$. Moreover, the set
$$
\mathfrak{m}_{m}^k=\{[f_1][f_2]\cdots[f_k]: [f_i]\in\mathfrak{C}_{m}^\infty(m)\},\, \mathfrak{m}_m=I 
$$
of all products of $k$ germs, locally, is the set of all germs which have null partial derivatives in $m$ to the order $k$, because if $f(m)=g(m)=0$, then $D|_m(fg)=f(m)D|_m(g)+D|_m(f)g(m)=0$. We have therefore the chain of ideals
$$
\ldots\triangleleft\mathfrak{m}_{m}^k\triangleleft\mathfrak{m}_{m}^{k-1}\triangleleft\ldots\triangleleft\mathfrak{m}_{m}\triangleleft\mathfrak{C}_{m}^\infty(M).
$$
We term the vector space $\mathfrak{m}_{m}/\mathfrak{m}_{m}^2$ \textit{cotangent vector space} which is isomorphic to $\mathbb{R}^n$, where $n$ is the dimension of the manifold. We construct the isomorphism as following. Fix a chart $\varphi$ such that $\varphi(m_0)=0$ (one can eventually translate the chart). We have the enveloping in $\varphi(m_0)$ of a map $f \in \mathcal{C}^\infty(M)$
$$
f(m)=f(m_0)+\sum_i\partial \varphi_i f(m_0)\varphi_i(m)+\sum_{ij}a_{ij}\varphi_i(m)\varphi_j(m)
$$
for some maps $a_{ij}$. Therefore, if $f(m_0)=0$, then
$$
[f]+\mathfrak{m}_{m_0}^2=\left[\sum_i\partial \varphi_i f(m_0)\varphi_i\right]+\mathfrak{m}_{m_0}^2,
$$
hence, the vector space isomorphism
$$
\sum_i\partial\varphi_i f(m_0)[\varphi_i]+\mathfrak{m}_{m_0}^2\in \mathfrak{m}_{m_0}/\mathfrak{m}_{m_0}^2 \mapsto 
\left(\partial\varphi_i f(m_0)\right)_{i=1,2,\ldots,n}\in\mathbb{R}^n.
$$
Taking a derivation $D_{|m}$ in $m$, since $D_{|m}(ff)=0$ when $f(m)=0$, the derivation induces a linear application $\mathfrak{m}_{m}/\mathfrak{m}_{m}^2\to\mathbb{R}$. Indeed, consider the map
$$
[\gamma]\in \Gamma_m\mapsto \tilde{D}_{\gamma}\in(\mathfrak{m}_{m_0}/\mathfrak{m}_{m_0}^2)^*
$$
where $D_\gamma$ is the derivation induced by $\gamma$ (equation (\ref{casertaderivation-induced-by-a-curve}), page \pageref{casertaderivation-induced-by-a-curve}) $\tilde{D}_{\gamma}[f]=[D_{\gamma(f)}]$. The map is well defined, since if $D_\gamma$ is a derivation in $m$ and $f(m)=g(m)=0$, then 
$$
D_{\gamma}([f]+[g_1][g_2])=D_{\gamma}([f + g_1g_2])=D_{\gamma}([f]),
$$
hence $D(fg)=0$.\\

In particular
$$
\mathfrak{m}_{m_0}/\mathfrak{m}_{m_0}^2\cong (\mathfrak{m}_{m_0}/\mathfrak{m}_{m_0}^2)^{**}\cong \Gamma_m^*.
$$
thus, we have the equivalence of the three definitions, in particular, the vector spaces $(I/I^2)^*$ and $\mathrm{Der}_m(M)$ are isomorphic. Let $I^2=\{fg:f,g\in I\}$ and consider the vector space $I/I^2$. For all $D\in\mathrm{Der}_m(M)$, $D(I^2)=0$, since, if $f$ is not trivial and constant, then there exists an open neighborhood $U$ of $m$ such that $f_{|U}=c\in\mathbb{R}^*$, thus
$$
cD_{|m}(f)=D_{|m}(cf)=D_{|m}(ff)=2cD_{|m}(f).
$$ 
The mapping
$$
\tilde{D}:I/I^2\to \mathbb{R}
$$ 
which maps $f+I^2$ to $D(f)+I^2$ is well defined and is a form. In particular $D\to\tilde{D}$ is a vector spaces isomorphisms, thus $\mathrm{Der}_m(M)$ and the dual space $(I/I^2)^*$ are canonically isomorphic.\\

We can make use of one of the three equivalent definitions for the \textit{tangent space} $T_m M$ to $M$ in $m\in M$, that is, one of the pairwise isomorphic vector spaces
$$
\mathrm{Der}_m(M)\cong(\mathcal{C}_m^\infty(M)/\mathcal{C}_{m,0}^\infty(M))^*\cong\Gamma_m.
$$

\subsubsection{Differential in a point} If $V,W$ are vector space and $f:V\to W$ is a homomorphism, then a homomorphism $f^t:W^*\to V^*$ between the dual spaces is defined by $f^t(\varphi):=\varphi\circ f$, $\varphi\in W^*$. Given two differential manifolds $M$ and $N$ and a map $f:M\to N$, one defines the algebra isomorphism $f^*:\mathcal{C}_m^\infty(N)\to\mathcal{C}_{f(m)}^\infty(M)$ by $f^*(h):=h\circ f$, that is,
\begin{enumerate}
\item $f^*(h+g)=(h+g)\circ f=h\circ f+ g\circ f$
\item $f^*(hg)=(hg)\circ f=(h\circ f)(g\circ f)$
\item $f^*(\lambda g)=\lambda g\circ f$
\item $f^*(h)=0$ if and only if $h\circ f(m)=0$ for all $m\in M$ ($h$ is null in some neighborhood of $f(m)$). 
\end{enumerate}

The \textit{differential in $m\in M$} of a map $f:M\to N$ between differential manifolds is the linear map $\mathrm{d}_mf:T_mM\to T_{f(m)}N$ which acts  as following
\begin{itemize}
\item if $D\in \mathrm{Der}_mM$, then $\mathrm{d}_mf(D):=D\circ f^*$;
\item if $g\in(\mathcal{C}_{m,0}^\infty(M)/\mathcal{C}_{m,0}^\infty(M)^2)^*$, then $\mathrm{d}_mf(g):=(\bar{f}^*)^t(g)$, where $\bar{f}^*$ is the application on the factor space induced by $f^*$;
\item if $\gamma\in\Gamma_m$, then $\mathrm{d}_mf([\gamma]):=[f\circ\gamma]$.
\end{itemize}

\subsection{Tangent bundle, vector fields and flows}
The global tangent space of $\mathbb{R}^n$ is just the direct product $\mathbb{R}^n\times \mathbb{R}^n$, that is, the set of all couple $(m,\vettore{v})$, where $m$ is a point in $\mathbb{R}^n$ and $\vettore{v}$ is an applied vector in $m$. The generalization of this construction to a manifold is made by means of a vector bundle, which is locally homeomorphic to a Cartesian product, globally (the trivial bundle) in the case $M=\mathbb{R}^n$ .

\subsubsection{Tangent bundle}   

Let $E$, $B$ be two topological spaces and $\pi:E\to B$ a subjective map. The triple $(B,E,\pi)$ is a \textit{fiber bundle} if there exists a topological space $F$, the fiber,  and an open subset $U$ of $B$ for all $b\in B$, such that $\pi^{-1}(b)$ is homeomorphic to $U \times F$. The local homeomorphisms are called \textit{local trivializations} and the fiber bundle is said \textit{trivial} if $E$ is globally homeomorphic to $B\times F$. If $B$ is a differential manifold, then it is possible to equip $E$ with a compatible atlas by lifting the charts of $B$.\\

For a differential manifold $M$ we have defined at each point $m\in M$ the tangent space $T_mM$; the collection of all these spaces indexed by the points of $M$ is the tangent space $TM$ of the manifold $M$. More precisely $TM$ is a vectorial fiber bundle with respect to the map $\pi:TM\to M$, $\pi((m,\underline{v}))=m\in M$, such that the fiber $\pi^{-1}(m)$ is diffeomorphic to $T_mM$. If $(U,\varphi)$ is a local chart, then the local trivialization is the map
$$
\tilde{\varphi}:\vettore{v}_m=(m,\underline{v})\in\pi^{-1}(U)\mapsto (\varphi(m),\mathrm{d}_m\varphi(\underline{v}))\in\mathbb{R}^n\times \mathbb{R}^n
$$
and it is also a chart which turns $TM$ into a $2n$-dimensional real manifold, where $\underline{v}$ is a vector of a tangent space in a point, a free vector, and $\vettore{v}_m=(m,\underline{v})$ is the vector $\underline{v}$ applied in $m$. The bundle $TM$ is therefore the disjoint union of all tangent spaces, thus, we can also give the equivalent definition: the \textit{tangent bundle} $TM$ to $M$ is the disjointed union of the tangent spaces
$$
TM=\dot{\bigcup}_{m\in M} T_mM=\{\vettore{v}_m:=(m,\underline{v}):m\in M, \underline{v}\in T_mM\}.
$$

The differential and the derivation in a point of manifold can be smoothly extended to the tangent bundle as following. Let $f:M\to N$ be a map between manifolds. The \textit{differential} of $f$ is the linear application $\mathrm{d}f$ between the tangent bundles, $\mathrm{d}f:TM\to TN$ defined by  
$$
\mathrm{d}f(\vettore{v}_m)=\mathrm{d}f((m,\underline{v}))=(f(m),\mathrm{d}_mf(\underline{v})).
$$

Let $f,g:M\to\mathbb{R}$ be smooth. A \textit{derivation} is a linear mapping $D:C^\infty(M)\to C^\infty(M)$ which maps functions onto functions such that
\begin{enumerate}
\item $D(fg)=fD(g)+D(f)g$ and
\item if $f,g$ take the same values on a open set, then $D(f)=D(g)$. 
\end{enumerate}
The set $\mathrm{Der}(M)$ of all derivation is an algebra since
\begin{itemize}
\item $D(f+g)=D(f)+D(g)$,
\item $(\lambda D)(f)=\lambda D(f)$
\end{itemize}
and a $C^\infty(M)$-modulo since
\begin{itemize}
\item $(D_1\circ D_2)(f):=D_1(D_2(f))$, $(D_1+D_2)(f)=D_1(f)+D_2(f)$,
\item $(hD)(f)=hD(f)$, $h\in C^\infty(M)$.
\end{itemize}
In local coordinates
$$
D(f)=\sum \partial x_i(f\circ\varphi^{-1})D(\varphi_i).
$$

\subsubsection{Vector field, flows and Lie derivative}

A \textit{vector field} $X$ is a collection of applied vectors, one for each point of the manifold, more precisely, it is a smooth section of the projection $\pi:TM\to M$, i.e., a smooth map $X:m\in M\to X(m)=(m,X_m)\in TM$ with $X_m\in T_mM$ for all $m\in M$. One denotes by $\mathfrak{X}(M)$ the set of vector fields, which is an infinite dimensional vector space and locally finite dimensional.\\

Vector fields and derivation reciprocally induce themselves, that is, if $m\in M$, the vector $X(m)=(m,[\gamma])$, $\gamma(0)=m$ of the bundle $TM$ induces the derivation
$$
D_X:f\mapsto \mathrm{d}/\mathrm{d}t|_{t=0}(f\circ\gamma).
$$
Conversely, if $D$ is a derivation and $\gamma(t)=(D(\varphi_1),D(\varphi_2),\ldots, D(\varphi_n))t+\varphi(m))$ is a curve, then 
$$
X_D(m)=(m,[\varphi^{-1}\circ\gamma])
$$
is vector field. Therefore, the action $X(f):m\mapsto D_{X(m)}(f)$ of a vector field $X$ on a map $f\in\mathcal{C}^\infty$is fully defined.\\

Since a vector field is a smooth section, locally, there exists a curve which has as the tangent vector in each point the vector field in that point, the so called  \textit{integral curve} of the vector field. That is, given the field $X$, a curve $\gamma:I\to M$ is integral curve of $X$ if 
$$
X(\gamma(t))=(\gamma(t),[\gamma\circ(\mathrm{id}+\bold{1}_t)])=(\gamma(t),\dot{\gamma}(t)).
$$
The set of all these integral curves together is the so called \textit{flow} induced by the vector field. More generally, a flow $\Phi:A\times M\subseteq\mathbb{R}\times M\to M$ is a map such that
\begin{itemize}
\item $\gamma_m(\cdot):=\Phi(\cdot,m):\mathbb{R}\to M$ are smooth curves for all $m\in M$,
\item $\varphi_t(\cdot):=\Phi(t,\cdot):M\to M$ are diffeomorphisms and
\item $t\mapsto \varphi_t$ is an additive group homomorphism.
\end{itemize} 

If $\pi:TM\to M$ is the tangent bundle and $\gamma:(a,b)\to M$ is a smooth curve, consider the lift $\tilde{\gamma}:(a,b)\to TM$ of $\gamma$ to $TM$, that is, $\pi\circ\tilde{\gamma}=\gamma$ or equivalently $\tilde{\gamma}(s)=(\gamma(s),[\gamma_{s+}])$ where $\gamma_{s+}(t):=\gamma(s+t)$. We say that $\gamma$ is a integral curve of $X$ if
$$
X(\gamma(s))=(\gamma(s),[\gamma_{s+}]),\quad s\in(-\varepsilon,\varepsilon)
$$
that is, for all $m\in M$, $\gamma_m$ solves the Cauchy problem
$$
\left\{\begin{array}{l}
X(\gamma(s))=(\gamma(s),[\gamma_{s+}]),\quad s\in(-\varepsilon,\varepsilon)\\
\gamma(0)=m
\end{array}
\right. 
$$
and $\Phi_X(t,m):=\gamma_m(t)$, with $m\in M$ and $t\in(-\varepsilon,\varepsilon)$, is the local flow induced by $X$. In particular, if the manifold is compact, flows and vector fields induce each other. Conversely, a flow $\Phi$ induces the vector field
$$
X_\Phi:m\in M\mapsto (m,[\gamma_m=\Phi(\cdot,m)])\in T_mM
$$
where
$$
\left\{
\begin{array}{l}
\Phi(0,m)=m\\\\
\partial\Phi/\partial\tau(t,m)=\dot{X}_{\Phi(t,m)}(0)
\end{array}
\right.
$$
for all $m$ and $t$, where $X_{\Phi(t,m)}$ is a curve by $\Phi(t,x)$.\\

Since vector fields and derivation also induce each other, also a flow $\Phi$ induces the derivation
$$
D_{\Phi}(f):=\partial \tau (f\circ\Phi)_{|\tau=t}\in\mathcal{C}^\infty(M).
$$

If $\Phi$ is a flow, then $t\mapsto\Phi(t,m)$ is a group homomorphism for each $m$: this is an example of \textit{one parameter group of diffeomorphisms}. It is a set of diffeomorphisms $\{\varphi_i\}_{i\in I}$ of $M$ in itself such that $t\mapsto\{\varphi_t\}$, $t\in\mathbb{R}$, is an epimorphism, that is, there exists $\phi:\mathbb{R}\to\mathrm{Diff}(M)$ which satisfies
\begin{enumerate}
\item $\phi(t_1+t_2)=\phi(t_1)\circ\phi(t_2)$
\item $(t,m)\in\mathbb{R}\times M\mapsto \Phi(t,m):=\varphi_t(m)=\phi(t)(m)\in M$ is smooth.
\end{enumerate}
The second property states that the diffeomorphisms vary smoothly with respect to both $t\in\mathbb{R}$ and $m\in M$. To each homomorphism is associated the curve $\gamma_{\varphi}$ through $m$
$$
\gamma_{\varphi}:t\in \mathbb{R}\to \varphi_t(m)\in M
$$
since $\gamma_{\varphi}(0)=\phi(0)(m)=\mathrm{id}_M(m)=m$. Thus, to each $\varphi$ is associated the vector field
$$
\frak{X}(M)\ni\dfrac{\partial\Phi}{\partial t}\Big|_0:m\in M\mapsto (m,[\gamma_{\varphi}])\sim (m,\gamma_\varphi'(0))\in TM
$$
and $\dfrac{\partial\Phi}{\partial t}\Big|_0$ is an integral curve since it satisfies
$$
\dfrac{\partial\Phi}{\partial t}\Big|_0(\gamma(t))=(\gamma(t),[\gamma(t)])\sim\gamma'(t)\,\quad t\in(a,b)
$$

One says that the one parameter group of diffeomorphisms $\{\phi_t:t\in\mathbb{R}\}$ induces the vector field $X$ if $X_m=\dot{\gamma}(0)$ where $\gamma(t)=\varphi_t(m)$ and the vector field is said \textit{complete} if the group is defined globally. Conversely, a vector field $X$ generates the local group $\{\varphi_t\}$ if it induces $X$; such a group always exists and it is always global if the manifold is compact.\\  

Let $X$ be a vector field which generates the group $\{\varphi_t\}$. If $\varphi:M\to M$ is a differentiable map, then the \textit{pushforward} $\varphi_*(X)$ of a vector field $X$ generates the group $\{\varphi\circ\varphi_t\circ\varphi^{-1}\}$ and $X$ is said $\varphi$-\textit{left invariant} if $\varphi_*(X)=X$, that is, $\varphi,\varphi_t$ commute.\\

Remember that to each field $X$ a derivation is associated
$$
X:f\mapsto \left(Xf:m\mapsto\mathrm{d}/\mathrm{d}t|_{t=0}(f\circ\gamma)\right)
$$
where $X(\gamma)=\dot{\gamma}(t)$ and $\gamma(0)=m$. Thus, if $X,Y$ are vector fields, $X(Yf)$ its defined and
\begin{multline*}
X(Y(fg))=X(f(Yg)+g(Yf))=\\
=fX(Yg)+(Xf)(Yg)+gX(Yf)+(Xg)(Yf)
\end{multline*}

since $X(f(Yg))=fX(Yg)+(Xf)(Yg)$. Therefore,
\begin{multline*}
X(Y(fg))-Y(X(fg))=fX(Yg)+gX(Yf) - fY(Xg)-gY(Xf)=\\
=f(X(Yg)- fY(Xg))+g(X(Yf)-Y(Xf))
\end{multline*}
thus, $[X,Y]:f\mapsto [X,Y]f:=X(Yf)-Y(Xf)$ is the derivation associated to the commutator $[X,Y]$. In local coordinates, if $X=\sum_i X_i\vettore{\varphi}_i$ and $Y=\sum_i Y_i\vettore{\varphi}_i$, then
$$
[X,Y]=\sum_{ij} \left(X_i(\vettore{\varphi}_iY_j)-Y_i(\vettore{\varphi}_iX_j)\right)\vettore{\varphi}_j
$$

Let $X$ be a complete field which generates the group $\{\varphi_t\}$, $\Phi$ the associated flow and $Y$ another vector field. Since $\varphi_t$ is a diffeomorphism we have the differential map $\mathrm{d}\varphi_t:TM\to TM$, 
$$
\mathrm{d}\varphi_t: \vettore{v}_m\in T_mM\mapsto \mathrm{d}_m\varphi_t(\vettore{v}_m)\in T_{\varphi_t(m)}M
$$
hence the vector field
$$
Y_t=(\varphi_t)_*(Y):=\mathrm{d}\varphi_t\circ Y\circ\varphi_t^{-1}:M\mapsto T_{\varphi_t(m)}M
$$
the so called \textit{pull-back} of $Y$. We term the total derivative with respect to $t$
$$
\mathcal{L}_X Y_m:=\underset{t\to 0}{\mathrm{lim}}\dfrac{(\varphi_t)_*(Y)_m-Y_m}{t}=
-\mathrm{d}/\mathrm{d}t|_{t=0}(\varphi_t)_*(Y)_m.
$$
the \textit{Lie derivative} of $Y$ with respect to $X$. Standard arguments show that $\mathcal{L}_XY=[X,Y]$, that is, that the Lie derivative is actually the commutator of the fields, and the map $X\mapsto\mathcal{L}_X
$ is an affine connection. In particular one has 
$$
\mathcal{L}_XY-\mathcal{L}_YX-[X,Y]=[X,Y]
$$
where $\mathcal{L}_XY-\mathcal{L}_YX-[X,Y]$ is the torsion. We will see that the space $\frak{X}(M)$ of vector fields is an infinite dimensional Lie algebra with the commutator $\mathcal{L}_XY=[X,Y]$.

\section{Elements of Lie theory, homogeneous spaces and principal bundles}

In late nineteenth century F. Klein (1849 - 1925) brought order out of chaos among new types of geometries by means of the identification of a geometrical object with its group of symmetries which he called \textit{Hauptgruppe} or \textit{principal group}. To each geometry on some set $M$ corresponds the group of transformations acting transitively on it which preserve a class of geometrical features.

In the same period S. Lie (1842-1899), a student of P. L. M. Sylow (1832-1918), made use of the group theory in the study of the differential equations theory following the successful path traced by E. Galois with the theory of algebraic equations. Also W. Killing (1847-1923) independently followed this approach for the non Euclidean geometry. The theory of Lie have been later developed by E. Cartan (1869-1951) with her Ph.D. thesis and H. Weyl (1885-1955). The \textit{Hauptgruppe} introduced by Klein is an example of a Lie group and we will make a wide use of it.\\

Lie theory concerns both differentiable groups and the structure of the tangent space to these groups. The so called three Lie's theorems clear up the connection between a Lie group and a Lie algebra.

\begin{Theorem}
There exists a one-to-one correspondence between the connected subgroups of a Lie group $G$ and the set of the subalgebras of the Lie algebra $\mathfrak{g}$ of $G$.
\end{Theorem}

\begin{Theorem}
Let $G_1$ and $G_2$ be Lie groups and let $\mathfrak{g}_1$ and $\mathfrak{g}_2$ be the associated Lie algebras. Then,  $\mathfrak{g}_1$ and $\mathfrak{g}_2$ are isomorphic if and only if $G_1$ and $G_2$ are locally analytically isomorphic.
\end{Theorem}

\begin{Theorem}[Ado's theorem]
Let $\mathfrak{g}$ be a Lie algebra on the field $K$, $K=\mathbb{R},\mathbb{C}$. Then, there exists an analytic simply connected group whose Lie algebra is isomorphic to $\mathfrak{g}$. 
\end{Theorem}

Thanks to the work of Klein and Lie among the other, a geometry is therefore completely described by three objects: a manifold $M$, a point and the principal group, a Lie group $G$ which acts smoothly on $M$. Klein's geometry is homogeneous, the elements of $G$ are symmetries and the action is transitive, thus the angles, the lengths, the lines or the collinearity, for instance, are all preserved by the action and the points cannot be distinguished anymore only by geometrical properties.\\

In this section we will redefine the objects we defined in the first part by means of the action of the Lie group of symmetries on the manifold. In particular, the manifold will be identified with a factor set of its Lie group of symmetries and the (canonical) connection will be defined also without metrical considerations on the manifold.

\subsection{Lie group, Lie algebra and the exponential map}\label{casertalie-group}

A \textit{Lie group} $G$ is both a differential manifold and a group such that the product and the inversion are both differentiable, that is, it is a group equipped with a differential structure such that the maps $(a,b)\in G\times G\mapsto ab\in G$ and $a\in G\mapsto a^{-1}\in G$ are both differentiable. A \textit{Lie algebra} $\frak{g}$ is a vector space with a bilinear operator, the \textit{Lie brackets} or \textit{Lie product}, which is nilpotent ($[X,X]=0$) and fulfills the Jacobian identity 
$$
[X,[Y,Z]]+[Y,[Z,X]]+[Z,[X,Y]]=0.
$$
The Lie algebra $\frak{g}$ of the Lie group $G$ is the set of all the \text{left invariant} vector fields  
$$
\frak{g}=\frak{X}(G)^G=\{X\in\frak{X}(M):(\mathrm{d}\lambda_g)_*(X)=X\,\forall g\in G\}
$$
where $\lambda_g$, $g\in G$, are the left translations of $G$, $\mathrm{d}\lambda_g:TG\to TG$, $g\in G$, the differentials
and 
$$
(\mathrm{d}\lambda_g)_*(X)\in\frak{X}(G)\mapsto \mathrm{d}\lambda_g\circ X\circ \lambda_g^{-1}\in\frak{X}(G).
$$
The algebra $\frak{g}$ is a subalgebra of $\frak{X}(G)$ and is isomorphic to $T_eG$ by the isomorphism 
$X\in\frak{g}\mapsto X_e\in T_eG$,  
$e$ being the neutral element of $G$.\\

Let $A\in\mathfrak{g}$ and $\{\phi_t\}$ the group generated by $A$. Since $A$ is left invariant, $\phi_t\circ\lambda_a=\lambda_a\circ\phi_t$, thus, $\phi_t(a)=(\phi_t\circ\lambda_a)(e)=a\phi_t(e)$. By setting $a_t:=\phi_t(e)$ one has 
$$
a_{t+s}=\phi_{t+s}(e)=\phi_t(\phi_s(e))=\phi_t(a_s)=a_s\phi_t(e)=a_sa_t
$$
thus, $\{a_t\}$ is a subgroup of $G$. Define therefore the \textit{exponential function} $\mathrm{exp}:\mathfrak{g}\to G$ by setting $\mathrm{exp}A:=g_1$. Let $k\in\mathbb{R}$ and $B$ the vector field induced by the group $\{\psi_t:=\phi_{kt}\}$, since $B_m=\mathrm{d}/\mathrm{d t}_{t=0}\phi_{kt}(m)=kA_m$, then  
$$
\mathrm{exp}(kA)=\mathrm{exp}(B)=\psi_1(e)=\phi_{k}(e)=a_k.
$$
The exponential function maps $\mathfrak{g}$ onto $G$ homeomorphicay. Moreover, if $\gamma:t\mapsto a_t=\phi_{t}(e)$ is a curve in $G$, then
$$
\dot{a}_k=\mathrm{d}/\mathrm{d}t_{t=0} (a_{t+k})=A_{a_k}=a_k A_e 
$$
hence, $a_t$ is the solution of the differential equation $\dot{a}_k=a_kA_e$.

\subsection{Homogeneous spaces and principal bundle}

Let $G$ be a Lie group, $B$ a differentiable manifold and assume that there exists a left action $(g,m)\in G\times B\mapsto gb\in B$ of $G$ on $B$ as a group of transformations. Also assume that the action is transitive, fix a base point $b_0\in B$ and consider the \textit{isotropy group} (of $b_0$) $H:=G_{b_0}$, the stabilizer of $b_0$ in the action as a group of diffeomorphisms. Denote by $\lambda:G/H\to M$ the map $\lambda:(gH,p)\mapsto g(p)=\rho_p(g)$. The group $H$ is a closed subgroup of $G$, hence a Lie subgroup and a submanifold of $G$. The space $G/H$ has a unique differential structure such that the projection $\pi:G\to G/H$ is smooth (and open) and $G$ acts transitively by left translation on $G/H$. Since the action is transitive, the map $gH\in G/H\mapsto gb_0\in B$ is subjective, in particular, if it is a homomorphism, it is also diffeomorphism.

\subsubsection{Principal bundle}

A fiber bundle $H\to G\overset{\pi}{\to} M$ is a \textit{principal fiber bundle}\index{principal bundle} with structure group $H$ if the Lie group $H$ acts freely on the right on $G$ and $M=G\backslash H$, where $G\backslash H$ is the set of the orbits $pH$, each of them being a fiber through some $p\in G$ over $m=\pi(p)$, that is,  
$$
p\in G_m=pH=\pi^{-1}(m)=\pi^{-1}(\pi(p)).
$$
The action of $H$ is fiber preserving and simply transitive on each fiber, the map $\phi_p:g\in H\to \phi_p(g):=pg\in H_m=pH$ is a diffeomorphism and is $H$-equivariant, that is, $\phi_p(gh)=\phi_p(g)h$ for all $h\in H$. In particular, if $G$ acts on $M$
$$
H\to G\stackrel{\pi}{\rightarrow} G/H\overset{\rho_p}{\cong} M
$$
is a principal bundle with structure group $H$ and fibers $gH$, $g\in G$. For the tangent space we have therefore
$$
T_eH\to T_eG\stackrel{\mathrm{d}_e\pi}{\rightarrow} T_H(G/H)\overset{\mathrm{d}_H\rho_p}{\rightarrow} T_pM.
$$
The tangent space $\mathcal{V}_p=T_p(G_m)$ to the fiber $G_m$ in $p$ is called \textit{vertical space} and the distribution $\mathcal{V}:p\in P\to\mathcal{V}_p\subseteq TP$ is the \textit{vertical distribution}.\\

If $\sigma:U\subseteq G/H\to G$ is a local analytic (cross) section, i.e., $\pi(\sigma(gH))=gH$, then $\sigma(U)$ is a submanifold of $G$ (but not a subgroup) which is diffeomorphic to  $\sigma(U)p\subseteq M$ according to the action of $G$ on $M$. Furthermore, there exists a one-to-one correspondence between the local sections and the local trivializations of the principal bundle, this allow us to identify portions of the manifold $M$ on which $G$ acts not only with the set of classes but also with the subsets of $G$. In particular, if $\sigma:U\to G$ is a local section, then
$$
\varphi:g\in\pi^{-1}(U)\subseteq G\to (gH,\sigma(gH)^{-1}g)\in U\times H
$$
is the associated local trivialization (a diffeomorphism), indeed
$$
\varphi(\pi^{-1}(U))=U\times H\to \sigma(U)\times H\subseteq M\times H
$$
with $\varphi(\sigma(gH))\to (\sigma(gH),1)\in \sigma(U)\times\{e\}$.\\

Let $\sigma:U\subseteq G/H\to G$ be a local section with $\sigma(H)=e$ and $L:=\sigma(U)$ (in particular $L\cap H=\{e\}$). If $\gamma:I\to M$ is a curve with $\gamma(0)=e$, then, for all $t\in I$ there exists $\tilde{\gamma}:I\subseteq\mathbb{R}\to G/H$ such that $\tilde{\gamma}(t)p=\gamma(t)$. Then 
$$
\mathrm{d}_e\rho_p|_{T_eL}:A=[\gamma]\in T_eL\leq T_eG\to Ap=[\gamma p]\in T_pM
$$ 
is a diffeomorphism. One says that $T_eG$ is \textit{reductive}, i.e., 
$$
T_eG=T_eL\oplus T_eH,
$$
$T_eH$ being the kernel of $\mathrm{d}_e\pi:T_eG\to T_H(G/H)$, if and only if $T_eL$ is $\mathrm{Ad}_H$-invariant, i.e., $L$ is $\mathrm{ad}_H$-invariant. More generally, given a homogeneous space $M=G/H$ and the Lie algebras $\mathfrak{g},\mathfrak{h}$ of $G$ and $H$ respectively, $\mathfrak{g}$ is called \textit{reductive}\index{reductive space} if there exists a vector space complement $\mathfrak{m}$ of $\mathfrak{h}$ in $\mathfrak{g}$, i.e., $\mathfrak{g}=\mathfrak{m}\oplus\mathfrak{h}$, which is $\mathrm{Ad}H$-invariant. We can identify $\mathfrak{m}$ with $T_pM$ since $\mathrm{ker}\,\mathrm{d}_e\pi=\mathfrak{h}$, whence $\mathrm{d}_e\pi(\mathfrak{g})=\mathfrak{m}$. Sufficient conditions for the space to be reductive are
\begin{enumerate}
\item the linear space $\mathrm{Ad}_{\mathfrak{g}}(H)$ is completely reducible,
\item $\mathfrak{g}$ has an $\mathrm{Ad}_{\mathfrak{g}}(H)$-invariant bilinear form which is not degenerating on $\mathfrak{h}$, f.i, the Killing bilinear form $K(u,v):=\mathrm{tr}(\mathrm{ad}_u\mathrm{ad}_v)$, $u,v\in M_n(K)\equiv\mathfrak{g}$.
\end{enumerate}

\subsubsection{Tangent bundle as principal bundle}

Let $TM=\dot{\cup}_{q\in M} T_qM$ be the tangent bundle of $M$. Since $\lambda_g:q\in M\mapsto gq\in M$, the group $G$ acts on $TM$ by $\mathrm{d}\lambda_g$. Consider the map
$$
(g, [\gamma])\in G\times T_pM\mapsto \mathrm{d}_p\lambda_g([\gamma])=[g\gamma]\in T_{gp}M\subseteq TM\mapsto gp\in M 
$$
then, $(g_1, [\gamma_1]),(g_2,[\gamma_2])$ have the same image, if and only if $g_2^{-1}g_1\in H$. Since $\gamma(t)=a_tp$ (the horizontal lift)
$$
g_2^{-1}g_1[\gamma_1]=[g_2^{-1}g_1a_{1,t}h_tp]=[\mathrm{ad}_{g_2^{-1}g_1}(a_{1,t})p]=\mathrm{Ad}_{g_2^{-1}g_1}([\gamma_1])=[\gamma_2]
$$
hence the class $[g,[\gamma]]$ is the set
$$
[g,[\gamma]]=\{(gh,\mathrm{Ad}_{h^{-1}}([\gamma]): h\in H\}.
$$
We can define therefore the \textit{associated bundle} $E=G\times_H T_eL$ as $G\times T_eL$ modulo the action
$$
(g,A)h:=(gh^{-1},\mathrm{Ad}_h(A)),\, h\in H
$$
with the projection $\pi_E:(g,A)H\in E=G\times_H T_eL\to gp\in M$. This bundle is itself a principal bundle with fiber
$$
\pi_E^{-1}(q=gp)=\{(g,A)h:h\in H,\, A\in T_eL\}=gH\times _H T_eL\equiv (g(p),T_pM)
$$
and structure group $H$.\\

Finally, if $H\to G\overset{\pi}{\to} M=G/H$ is a principal bundle, then the tangent bundle $TM$ is diffeomorphic to the associate bundle $G\times_H\mathfrak{m}$ by
$$
[(g,X)]\in G\times_H\mathfrak{m}\mapsto(\mathrm{d}_e\pi\circ\mathrm{d}\lambda_g)(X)\in T(M).
$$
where the quotient $G\times_H\mathfrak{m}$ is with respect to the action
$$
h:(g,X)\mapsto(gh^{-1},\mathrm{Ad}_h(X)).
$$

\section{Connection, parallel transport, covariant derivative and exponential map}
The two ingredients one needs to define the composition law given by Kikkawa are the exponential map and the parallel transport of a vector (field) along a curve: the first allows us to sum points of a space by means of vectors in the tangent bundle, the latter makes it possible to control somehow the variation of a vector field with respect to some constant field. Therefore we need a way to derive the vector fields in the tangent bundles in order to compute their variations, that is, in some sense, to ascertain what fields are constant. One of the tool we know already is the Lie derivative. However, it is not just a function of the vector field to be derived but also of the vector field with respect of which one derives. We need instead a tensorial derivative: such a derivative is the so called \textit{covariant derivative} and it will allow us to define the classes of parallel vectors and the \textit{parallel transport} of a vector along a curve. This covariant derivative is defined by means of the so called \textit{connection}, which we will define below as linear map and later by means of the algebraic split of the Lie algebra of the manifold. 

\subsection{Affine and canonical connections}
We begin by defining an affine connection, \textit{affine} because, under given conditions, the space of all connections is actually an affine space. Connections, covariant derivative and parallel transport are equivalent ideas: one can define one out of the three and deduce the other two.

\subsubsection{Affine connection}
An \textit{affine connection} is a sort of collection of covariant derivatives, it is a mapping $\nabla$ 
$$
\nabla:X\in \mathfrak{X}(M)\mapsto \nabla_X\in\mathrm{Hom}(\frak{X}(M),\frak{X}(M))
$$
which maps a vector field $X$ to a linear operator $\nabla_X$
such that
\begin{enumerate}
\item $\nabla_XY$ is $\mathbb{R}$-linear on $X$ and $Y$,
\item $\nabla$ is tensorial in $X$, $\nabla_{fX}Y=f\nabla_XY$,
\item $\nabla_XY$ is a derivation in $Y$, $\nabla_X(fY)=(Xf)Y+f\nabla_XY$ and
\item $\nabla_{fX+gY}=f\nabla_X+g\nabla_Y$,
\end{enumerate}
$f$ being smooth. A connection can be expressed locally by the so called \textit{Christoffel symbols}, that is, a set of maps $\Gamma_{ij}^k\in\mathcal{C}^\infty(U)$ which fulfill
$$
\nabla_{\vettore{\varphi}_i}\,\vettore{\varphi}_j=\sum_{k=1}^n\Gamma_{ij}^k \vettore{\varphi}_k,\quad i,j=1,2,\ldots, n,
$$
where $\{\vettore{\varphi}_i\}$ are the elementary vector fields.\\

As an example of connection we give account of that one  induced by a metric on the manifold and afterwards that one induced by a horizontal and vertical distribution.

\textbf{Connection induced by a metric}
A Riemannian metric on differentiable manifold $M$ is a section $g$ of the fiber bundle of the symmetric bilinear forms defined on $TM$ such that the metric $g_m$ in $m\in M$ is a definite positive form for all $m\in M$. Such a metric is a tensor which maps a couple of vectors to a number, a tensor $g\in\mathcal{C}^\infty(M,T^0_2)$, $0$-times contravariant and $2$-times covariant, which fulfills
\begin{enumerate}
\item $g(X,Y)=g(Y,X)$ for all $X,Y\in\frak{X}(M)$,
\item $g(X,X)(m)\geq 0$ for all $X\in\frak{X}(M)$ with $X(m)\ne 0$,
\end{enumerate}
where $g(X,Y)(m):=g_m(X_m,Y_m)$. In local coordinates, setting
$$
g_{ij}(m):=g\left(\dfrac{\partial}{\partial \varphi_i},\dfrac{\partial}{\partial \varphi_j}\right)(m)=g_m\left(\dfrac{\partial}{\partial \varphi_i}\bigg|_m,\dfrac{\partial}{\partial \varphi_j}\bigg|_m\right)
$$
we have
$$
g=\sum_{ij} g_{ij}\mathrm{d}\varphi^i\otimes\mathrm{d}\varphi^j,
$$
that is, for the fields $X=\sum X(\varphi_i)\partial/\partial \varphi_i$ and $Y=\sum Y(\varphi_i)\partial/\partial \varphi_i$, one has
\begin{multline*}
g(X,Y)(m)=\sum_{ij} g_{ij}(m)\mathrm{d}_m\varphi^i\otimes\mathrm{d}_m\varphi^j
\left(X(\varphi_i)(m)\dfrac{\partial}{\partial \varphi_i}\bigg|_m,Y(\varphi_j)(m)\dfrac{\partial}{\partial \varphi_j}\bigg|_m\right)=\\
=\sum_{ij} g_{ij}(m)X(\varphi_i)(m)Y(\varphi_j)(m).
\end{multline*}

The fundamental theorem of Riemannian geometry states that there exists a unique torsion free connection which preserves the metric (the \textit{Levi-Civita} affine connection), that is, 
\begin{enumerate}
\item $\nabla g = 0$ and 
\item $\nabla_XY - \nabla_YX = [X,Y]$, for all vector fields $X$  and $Y$.
\end{enumerate}

\textbf{From affine connection to covariant derivative}
If $X,Y$ are fields with $X=\sum_{i=1}^n X\varphi_i\vettore{\varphi}_i$ and $Y=\sum_{i=1}^n Y\varphi_i\vettore{\varphi}_i$ , then
$$
\nabla_XY=\nabla_X\left(\sum_i Y\varphi_j\dfrac{\partial}{\partial\varphi_j}\right)=
\sum_{i=1}^n X(Y\varphi_i)\dfrac{\partial}{\partial\varphi_j}+Y\varphi_j\nabla_X\left(\dfrac{\partial}{\partial\varphi_j}\right),
$$
where
$$
\nabla_X\left(\dfrac{\partial}{\partial\varphi_j}\right)=\nabla_{\sum_iX\varphi_i\partial/\partial\varphi_i}\left(\dfrac{\partial}{\partial\varphi_j}\right)=\sum_{i=1}^n (X\varphi_i)\nabla_{{\partial}/{\partial\varphi_i}}\left(\dfrac{\partial}{\partial\varphi_j}\right),
$$
whence
$$
\nabla_XY=\sum_{k=1}^n\left(X(Y\varphi_k)+\sum_{i,j}\Gamma_{ij}^k(X\varphi_i)(Y\varphi_j)\right)\dfrac{\partial}{\partial\varphi_k}.
$$
Since for a field $X=\sum_i X\varphi_i\,\vettore{\varphi}_i$ and an integral curve $\gamma$ through $m=\gamma(0)$ one has 
\begin{multline*}
X_m=X(m)=\sum_i (X\varphi_i)(m)\,\vettore{\varphi}_i(m)=\sum_i\dfrac{\mathrm{d}(\varphi_i\circ\gamma)}{\mathrm{d}t}\bigg|_{t=0}[\varphi^{-1}(\underline{e}_it+\varphi(m))]=\\
=\sum_i\dot{\gamma}_\varphi(0)_i[\varphi^{-1}(\underline{e}_it+\varphi(m))]=\dot{\gamma}(0)
\end{multline*}
we can compute
$$
\nabla_XY(m)=\sum_{k=1}^n\left(\dfrac{\mathrm{d}((Y\varphi_k)\circ\gamma)}{\mathrm{d}t}\bigg|_{t=0}+\sum_{i,j}\Gamma_{ij}^k\dot{\gamma}_\varphi(0)_i(Y\varphi_j)(m)\right)[\varphi^{-1}(\underline{e}_kt+\varphi(m))]
$$
Finally, by taking the vector field $\tilde{Y}(t)$ along $\gamma$, $\tilde{Y}(t)=Y(\gamma(t))=\sum a_i(\gamma(t))\vettore{\varphi}_i(\gamma(t))=\sum Y(t)\tilde{\varphi}_i(t)$, we have the covariant differential 
$$
\nabla_{\dot{\gamma}}Y=\sum_{k=1}^n\left(\dot{Y}_k+\sum_{i,j}\Gamma_{ij}^k\dot{\gamma}_i Y_j\right)\vettore{\varphi}_k
$$
that is
$$
\tilde{Y}(t)=Y_m+\nabla_{\dot{\gamma}(0)}\tilde{Y}t+o(t).
$$

\textbf{From covariant derivative to parallel transport}
We say that a vector field $\vettore{u}=Y$ is parallel along the integral curve $\gamma$ of the vector field $\vettore{v}$ if $\nabla_{\dot{\gamma}}Y=0$ for all $t$ and we say that $\vettore{u}$ is parallel with respect to $\vettore{v}$ if $\nabla_{\vettore{v}}\vettore{u}=0$. Therefore, a curve $\gamma$ is a geodesic if $\nabla_{\dot{\gamma}}\dot{\gamma}=0$, that is 
$$
\nabla_{\dot{\gamma}}\dot{\gamma}=\sum_{k=1}^n\left(\dot{\gamma}_k+\sum_{i,j}\Gamma_{ij}^k\dot{\gamma}_i \dot{\gamma}_j\right)\vettore{\varphi}_k=0.
$$

If $p,q\in M$ and $\underline{v}\in T_pM$ a fixed vector, then there exists a unique geodesic $\gamma$ with $\gamma(0)=p$ and $\dot{\gamma}(0)=\underline{v}$ and a unique parallel field $Y$ with respect to  $\nabla_X$ such that $Y(p)=v$. The map $v\mapsto Y(\gamma(t))$ is a the parallel transport from $p$ to $\gamma(t)$ along $\gamma$ and it is also a vector space isomorphism.

\textbf{From parallel transport to covariant derivative}
For each curve $\gamma$, the set of isomorphisms 
$$
\tau_{\gamma,s,t}:T_{\gamma(s)}M\to T_{\gamma(t)}M
$$
are called \textit{parallel transport} if they fulfill
\begin{enumerate}
\item $\tau_{\gamma, s,t}$ is smooth with respect to both $s$ and $t$ (and $\gamma$),
\item $\tau_{\gamma, s,s}=\mathrm{id}$ and
\item $\tau_{\gamma, u,t}\circ\tau_{\gamma, s,u}=\tau_{\gamma, s,t}$
\end{enumerate}
If such a parallel transport is given, then also the so called \textit{covariant derivative} $\nabla$ in $p$ along $\gamma$ of the field $\vettore{u}$ can be defined by
$$
\nabla_{\dot{\gamma}(0)}\vettore{u}=\underset{t\to 0}{\mathrm{lim}}\dfrac{\tau_{\gamma,0,t}^{-1}\left(\vettore{u}_{\gamma(t)}\right)-\vettore{u}_m}{t}\in T_mM,\quad m=\gamma(0).
$$
We can also define the geodesic curves associated to the parallel transport as the class of curves which satisfy $\tau_{\gamma,s,t}(\dot{\gamma}(s))=\dot{\gamma}(t)$ and also the associated field is given by  
$$
\nabla_{\vettore{v}}\, \vettore{u}:m\mapsto (\nabla_{\vettore{v}}\, \vettore{u})(m)=\nabla_{\vettore{v}_m}\, \vettore{u}=\nabla_{\dot{\gamma}(0)}\, \vettore{u},
$$
where $\gamma$ is a geodesic such that $\dot{\gamma}(0)=\vettore{v}(m)$. 

\subsection{Canonical connection on a principal bundle} 

Let $G\to P\overset{\pi}{\to} M$ be a principal bundle. A \textit{connection} on $P$ is a distribution $\mathcal{H}:p\in P\to T(P)$, the horizontal distribution, of the subspaces $\mathcal{H}_p\leq T_p(P)$, the horizontal spaces, such that
\begin{itemize}
\item $T_p(P)=\mathcal{V}_p\oplus\mathcal{H}_p$ and
\item $(\rho_g)_*(\mathcal{H}_p)=\mathcal{H}_{pg}$ for all $g\in G$.
\end{itemize}

A \textit{canonical connection}\index{canonical connection} is given on $M$ by defining the vertical and horizontal distribution as following. The vertical space in $g\in G$  is the tangent space of the fiber $gH$, that is
$$
\mathcal{V}_g:=T_g(gH)=\mathrm{d}_e\lambda_g(T_eH)
$$
and the horizontal space is
$$
\mathcal{H}_g:=\mathrm{d}_e\lambda_g(T_eL)=T_g(gL).
$$
$L$ being a set of representatives of $G/H$. For $[\gamma]=[ga_t]\in\mathcal{H}_{g}$, one has
$$
\mathrm{d}_g\rho_h([\gamma])=[ga_th])=[gha_t^{h^{-1}}]\in T_{gh}gh\,\mathrm{ad}_{h^{-1}}(L)=T_{gh}(ghL)\in\mathcal{H}_{gh}
$$
hence  the required condition $\mathrm{d}_g\rho_h(\mathcal{H}_g)=\mathcal{H}_{gh}$ is fulfilled.\\

We can extend the connection to $TM$ by mean of the associated bundle $E=G\times_H T_eL$ on $M$ with fibers $\pi_E^{-1}(q=gp)=gH\times T_eL$ and structure group $H$
$$
H\times_H T_eL\to E=G\times_H T_eL\stackrel{\pi}{\rightarrow} M.
$$
The vertical space in $w=(g,A)H\in E$ is the tangent space to the fiber though $w$ and the horizontal space is defined as $\mathcal{H}_{(g,A)H}=gH\times_H T_eL$, that is
$$
\mathcal{H}_{(g,A)H}=\{(g_1,A_1)H:g_1\in G,A_1\in \mathcal{H}_g\}.
$$

\subsubsection{Parallel displacement and covariant derivative}
Let $\gamma$ be a differentiable curve in the base manifold $M$ and let $\gamma^*$ be a horizontal lift of $\gamma$ to $P$, that is, a curve in $P$ such that $\pi\circ\gamma^*=\gamma$ and $\dot{\gamma}(t)\in\mathcal{H}_{\gamma(t)}$, which is unique if we fix $p\in P$ such that $\pi(p)=\gamma(0)$. The horizontal lift $\gamma^*$ is the \textit{parallel displacement} of the point $\gamma(0)=x_0$ in the fiber $\pi^{-1}(m_0)$ to the point $\gamma(1)=x_1$ in the fiber $\pi^{-1}(m_1)$ along the curve $\gamma$. We can proceed in a same way with vector fields. Let $X$ be a vector field on $M$, there exists a unique horizontal lift $X^*$ of $X$ in $P$, that is, $\tilde{\pi}(X^*_u)=X_{\pi(u)}$ where $\tilde{\pi}=\mathrm{d}_u\pi:T_u(P)\to T_{\pi(u)}(M)$ with $\tilde{\pi}|_{\mathcal{H}_u}:\mathcal{H}_u\to T_{\pi(u)}(M)$ being an isomorphism.\\ 

Once a connection is given, we can define the \textit{parallel displacement} of the fibers $gH$. If $\gamma:I\to M$ is a smooth curve, for all $y\in \pi^{-1}(\gamma(0))$, there exists a unique horizontal lift $\gamma^*_y:I\to G$ of $\gamma$ through $y$, i.e., 
$$
\pi(\gamma^*_y(t))=\gamma^*_y(t)p=\gamma(t),\, [\gamma]\in T_eL.
$$
Actually, if $\gamma:I\to M$, since the bundle is locally trivial, there exist $a_t\in L$ with $a_0h_0=y$ such that $\gamma(t)=a_th_0
p$, thus, $\gamma^*_y(t)=a_th_0$ is horizontal and $\pi(\gamma^*_y(t))=a_tH=\gamma(t)$. 
Define the parallel displacement along $\gamma$ by 
$$
\tau_{0,t}:y=a_0h_0\in \pi^{-1}(a_0H)\mapsto a_th_0\in \pi^{-1}(a_tH).
$$
We can also define the parallel displacement in the tangent bundle. Let $\gamma:I\to M$, $\gamma(0)=gp$, and let $\gamma^*=ga_t$ be the horizontal lift to $G$, thus, for all $A\in T_eL$, $(ga_t,A)H$ is a horizontal lift to $E$ (cfr. \cite{kobayashi1963foundations} page 114). The parallel transport
$$
\tau_{t,t+\varepsilon}:\pi_E^{-1}(\gamma(t+\varepsilon))\to \pi_E^{-1}(\gamma(t))
$$
is therefore 
$$
\tau_{t,t+\varepsilon}:(ga_{t+\varepsilon},A)H\in \pi_E^{-1}(\gamma(t+\varepsilon))\mapsto(ga_t,A)H\in\pi_E^{-1}(\gamma(t)).
$$

\textbf{Covariant derivative}
Let $X:E\to M$ be a section, i.e., $\pi_E(X_{\gamma(t)})=\gamma(t)$ (a vector field) and let $\gamma:I\to M$. Define the covariant derivative $\nabla_{\dot{\gamma}(t)}X\in \pi_E^{-1}(\gamma(t))$ of $X$ in the direction (with respect to) $\dot{\gamma}$ by
$$
\nabla_{\dot{\gamma}(t)}X:=\underset{h\to 0}{\mathrm{lim}}\dfrac{\tau_{t,t+\varepsilon}\left(X_{\gamma(t+\varepsilon)}\right)-X_{\gamma(t)}}{\varepsilon}
$$
and the curve $X_{\gamma(t)}$ in $E$ is said to be parallel if $\nabla_{\dot{\gamma}(t)}X=0$ (see \cite{kobayashi1963foundations} page   124), in particular
$$
\nabla_{\dot{a}_t}\dot{a}_t=\underset{h\to 0}{\mathrm{lim}}\dfrac{\tau_{t,t+\varepsilon}\left(X_{\gamma(t+\varepsilon)}\right)-X_{\gamma(t)}}{\varepsilon}=0.
$$ 

\subsubsection{Canonical connection for reductive space}

Let $M$ be homogeneous and reductive. The \textit{canonical connection} on $M$ is the connection on the principal bundle $H\to G\overset{\pi}{\to} G/H\cong M$ associated to the distributions $\mathcal{H}_g:=(\lambda_g)(\mathfrak{m})$ and  $\mathcal{V}_g:=(\lambda_g)(\mathfrak{h})$, $g\in G$, which are both provided by the following construction.

Since the projection $\pi:G\to G/H$ is smooth, the differential (the total derivative of $\pi$ or pushforward) $\pi_*=\mathrm{d}_p\pi:T_pG\to T_{pH}(G/H)$ is linear. One defines the vertical space $V_pG=\mathrm{ker}\mathrm{d}_p\pi$. Let $\Phi:(g,h)\in G\times H\to \Phi_h(g)=gh$ be the right action of the bundle, a connection on the principal bundle is therefore given by the (smooth) distributions $V_pG,H_pG$ which satisfy
\begin{itemize}
\item $T_pG=V_pG\oplus H_pG$, equivalently, $\pi_*(H_pG)=T_{\pi(p)}G$,
\item $(\Phi_h)_*(H_pG)=H_{\Phi_h(p)}G$
\end{itemize}

As an example of a homogeneous reductive space we have the \textit{symmetric space} for a connected Lie group $G$, the homogeneous space $G/H$ where $H$ is the stabilizer of a point and an open subgroup of $G^\sigma=\{g\in G:\sigma(g)=g\}$, $\sigma$ being an involution. The differential $\mathrm{d}_e\sigma$ of $\sigma$ in $e\in G$ is an involutive endomorphism of the Lie algebra $\mathfrak{g}$, thus, if $\lambda$ is an eigenvalue, then $\mathrm{d}_e\sigma(v)=\lambda v$ implies $v=\mathrm{d}_e\sigma^2(v)=\lambda\mathrm{d}_e\sigma(v)=\lambda^2v$, hence, $\lambda=\pm 1$. The eigenspaces are therefore $\mathfrak{h}$, the Lie algebra of $H\leq G^\sigma$, which is stable under $\mathrm{d}_e\sigma$ and the eigenspace of $-1$, which is the complement $\mathfrak{m}$ of $\mathfrak{h}$ in $\mathfrak{g}$, $\mathfrak{g}=\mathfrak{h}\oplus\mathfrak{m}$, with $[\mathfrak{h},\mathfrak{h}]\subseteq\mathfrak{h}$, $[\mathfrak{h},\mathfrak{m}]\subseteq\mathfrak{m}$ and $[\mathfrak{m},\mathfrak{m}]\subseteq\mathfrak{h}$.

\subsection{Exponential map} We give now account of the other ingredient of Kikkawa's composition law, the exponential map from the tangent space to the manifold.\\

Let $M$ be a manifold and $p\in M$ a fixed point. There exists a neighborhood $N_{\underline{0}} (p)$ of $\underline{0}$ in $T_pM$ and a neighborhood $U(p)$ of $p$ such that for all $\underline{v}\in T_pM$ 
\begin{itemize}
\item the geodesic $\gamma_{\underline{v}}$ is defined on some set $J_{\underline{v}}\supseteq [0,1]$, 
\item $\gamma_{\underline{v}}(1)\in U(p)$ and
\item $\mathrm{exp}_p:\underline{v}\in N_{\underline{0}} (p)\mapsto\gamma_{\underline{v}}(1)\in U(p)$ is a diffeomorphism.
\end{itemize}
Consider $\mathrm{exp}:p\in M\mapsto\mathrm{exp}_p$ as a function of $p\in M$. Fixing $p\in M$, a star-shaped neighbourhood $N_{\underline{0}}(p)$ of $\underline{0}\in T_pM$where $\mathrm{exp}_p$ is a diffeomorphism onto an open subset $N_p$ of $M$ is called \textit{normal} and $N_p=\mathrm{exp}_p(N_0(p))$ is called \emph{normal neighbourhood} of $p$. Assume that each point has a normal neighbourhood, which is also normal for all other points of the set. If furthermore, for $q_1,q_2\in N_p$, there exists a unique geodesic $\gamma:[0,1]\to N_p$ with $\gamma(0)=q_1$ and $\gamma(1)=q_2$, then $N_p$ is called \textit{convex}, \textit{simple} if the geodesic is unique. One can show that in a manifold with an affine connection each point has a system of simple, connected and open neighborhoods. Furthermore, the exponential map is differentiable and $\mathrm{d}_0\mathrm{exp}_p$ is ``almost'' the parallel transport.

\section{Differential loop of Lie type}
Loop theory is relatively young and covers various areas of mathematics such geometric algebra, topology and combinatorics\footnote{We recommend two very interesting works, Historical notes on loop theory by H. O. Pflugfelder, \cite{pflugfelder2000historical} and Smooth Quasigroups and Loops forty-five years of incredible growth by L. V. Sabinin, \cite{sabinin2000smooth}, which provide detailed accounts of the progress of the loop theory during the XX century.}.

The statement which introduces a loop as an algebraic structure in the simplest way is ``a loop is a group without associativity'', which is also true but not completely. The theoretical core of loops and quasigroups theory is of course the non-associativity, which geometrically often arises in non-Euclidean context. As a first approach to the theory, the two works \textit{Moufang loops and Bol loops: Zur Struktur von Alternativkoerpern} by Ruth Moufang (1935)\footnote{R. Moufang (1905-1977) had studied at the University of Frankfurt, and later in K{\"o}nigsberg, where she was strongly influenced by Reidemeister.}, and \textit{Gewebe und Gruppen} by Gerrit Bol (1937) are always a good starting point; according to Pflugfelder, they ``marked the formal beginning of loop theory''.

Beside the term quasigroup, the word \textit{loop} designates those quasigroups with an identity. It seems that the invention of ``loops'' occurred around 1942 and apparently is due to the School of Chicago, where ``loop'' has been coined after the Chicago Loop, the elevated train over the main business area that actually forms a loop shaped path. The first publications introducing the term ``loop'' were the papers by Albert in 1943, \textit{Quasigroups I} and \textit{Quasigroups II} and afterwards the two publications by R. H. Bruck \textit{Some results in the Theory of Quasigroups} (1944) and \textit{Contributions to the Theory of Loops} (1946).
According to Pflugfelder, R. Baer and M. Hall were the main authors in the branch of geometry related to loops in the United States during the 1940s, both making use of the F. Klein's group theoretical approach.\\

In 1964 M. Kikkawa (see \cite{kikkawa1964local}) proved that it is always possible to define a loop operation among the points of manifold  with an affine connection within a suitable neighborhood of any fixed point, the so-called geodesic loops or local loop. This non-associative operation is the generalization of the applied vectors sum on a Euclidean affine plane, the parallelogram law, which is of course associative and commutative, to a Riemannian surface. Indeed, if we move from a flat plane to a non-flat manifold, it is well-known that the curvature of the space may causes a vector displacement after a parallel transport of a vector along a closed path. We will illustrate how to construct a loop of Kikkawa type on a $p$-adic $2$-sphere (for the real case see for instance \cite{nesterov2000non}).

\subsection{Loops and local loops} 
There are at least two different but related ways to build a (left) loop, we will give account here of the construction by means of (Lie) groups and transversals. In this sense we follow a flow which began with Baer \cite{baer1939nets} and involves many authors such as Karzel, Kreuzer, Strambach, Wefelscheid\footnote{see \cite{karzel1993groups}, \cite{kikkawa1975geometryofhomogeneous}, \cite{kikkawa1964local}, \cite{kikkawa2004kikkawa}, \cite{kikkawa1975geometryofhomogeneous}, \cite{kreuzer1994k},\cite{kreuzer1998inner}, \cite{nagy1994loops}}, Kikkawa, Sabinin, Ungar\footnote{see \cite{chein1990quasigroupsand}, \cite{friedman1994gyrosemidirect}, \cite{ungar1994holomorphic}}, Drapal, Kepka, Niemenmaa and Phillips\footnote{see \cite{drapal1993multiplication}, \cite{drapal1994multiplication}, \cite{kepka1993loops}, \cite{kepka1997connected}, \cite{niemenmaa1995loops}, \cite{niemenmaa1997connected}, \cite{niemenmaa1990multiplication}, \cite{niemenmaa1994connected}, \cite{phillips1999quasiprimitivity}, see also \cite{lgtlt}, pp. 21-22, \cite{bruck1966asurvey}, \cite{chein1990quasigroupsand}, \cite{foguel2000groupstransversal}, \cite{kiechle2002theoryofk}
\cite{lal1996transversalsin}, \cite{niemenmaa1990onmultiplication}}

\subsubsection{Quasigroups and loops}
A \textit{quasigroup} $(Q,\oplus)$, following to the definition of Moufang's paper\footnote{Actually, Moufang defined the quasigroups which today have her name, which are quasigroup satisfying any one of the Moufang identities. Moufang also proves that Q is diassociative, that is, the subquasigroup generated by any two elements is associative.}
, is an algebraic structure which consists of a non empty set $Q$ and a (non-associative) binary operation such that both the equations $a\oplus x = b$, $x\oplus a = b$ have a unique solution in $Q$. We term a quasigroup \textit{loop} if the quasigroup has a neutral element $e$, i.e. an element which satisfies both identities $a\oplus e = e\oplus a = a$. We have a left-loop (right-loop) in case that the operation has a neutral element and only the equation $a\oplus x=b$ ($x\oplus a=b$) has a unique solution.\\

For a given quasigroup $(Q,\oplus)$, the bijective maps $\lambda_a:x\mapsto a\oplus x$, $\rho_b:x\mapsto x\oplus b$, $a,b,x\in L$, are called \textit{left} and \textit{right translation} respectively. If $(Q,\oplus)$ is a left-loop, then only the left translations are bijective, the group $G(Q)$ generated by them is called \textit{enveloping group} of the left-loop and the stabilizers in $G(Q)$ of the neutral element $e$ with respect to the natural action of $G(Q)$ on $Q$ is generated by the maps $\delta_{a,b}:=\lambda_{a\oplus b}^{-1}\lambda_a\lambda_b$, $a,b\in Q$, which are called \textit{left inner deviations}.\\

A set $L$ endowed with a binary operation and a distinguished element $e\in L$ is called \textit{loop} if both left and right equations $a\cdot x=b$, $y\cdot a=b$ have a unique pair of solutions $x,y\in L$ for all $a,b\in L$ and $e$ is the neutral element. \\

Given a group $G$, a non-normal subgroup $H$ and a
left-transversal $L$ containing the identity of $G$, it is
elementary to see that $L$ is endowed with the structure of a
left-loop, denoted (even in the non-commutative case) by $(L,+)$.
Conversely, any left-loop can be identified with a
left-transversal of the group generated by the left translations
$x\mapsto a+x$ modulo the stabilizer of the identity. According to a consolidated notation (see \cite{lgtlt} page 17), we call a map $\sigma:G/H\to G$ a \textit{section}, if $(\pi\circ\sigma)(gH)=gH$ for all $g\in G$, where $\pi$ is the canonical projection. It follows that $\sigma(G/H)$ is a left-transversal, hence we consider the well known correspondence between left-loops and triples
$\left(G,H,\sigma(G/H)\right)$, where $\sigma$ is a section satisfying $\sigma(H)=1$. Without any great loss of
generality, one can assume that $G$ is generated by
$L:=\sigma(G/H)$ and that the action of $G$ on the homogeneous
space $G/H$, $aH\mapsto gaH$, is faithful. In this case the
\emph{enveloping group} $G$ decomposes as a quasi semi-direct
product $G=LH$, where for $a_i\in L$ and $h_i\in H$ $(i=1,2)$ the
multiplication in $G$ is defined by
$$
a_1h_1\,a_2h_2:=a_1+\delta_{h_1}(a_2)\,
\left(a_1,\delta_{h_1}(a_2)\right)\delta_{h_1}(a_2)^{-1}a_2^{h_1}h_1h_2
$$
where $\left(a_1,\delta_{h_1}(a_2)\right), \delta_{h_1}(a_2)^{-1}a_2^{h_1}\in H$ and $\delta_h:L\to L$ satisfies $\delta_h(x+y)=\delta_h(x)+\delta_{\mu(h)}(y)$ for a suitable map $\mu:H\to H$. If $\delta_h$ is also an automorphism of $L$, the left-loop $L$ is said to be an \emph{$A_l$-left-loop}.\\

 Actually, a left-transversal $L$ is a loop
precisely in the case where it is simultaneously a left-transversal
of all the homogeneous spaces $G/H^g$, $g\in G$, as proved by Baer
in \cite{baer1939nets} (cf. also \cite{lgtlt}, Proposition 1.6, page 18).\\

\begin{Proposition}\label{casertaG e G(L)}
For any faithful left-envelope $(G,H,L)$, there exists
an isomorphism $\lambda:G\to G(L)$, which maps $H$ onto $H(L)$ such that the restriction of $\lambda$ on $L$ onto $T(L)$ is a left-loop isomorphism. Furthermore, setting $\lambda_g:=\lambda(g)$ for all $g\in G$, one has $\{\lambda_h(a)\}=a^hH\cap L$ for all $h\in H$ and $a\in L$. 
\end{Proposition}

Given a left-enveloper $(G,H,L)$ the left-transversal $L$ is therefore a left-loop with respect the operation defined on it by $a+b:=c$ for $a,b\in L$, $c$ being the representative in $L$ of the class $abH$.

Two left-folders $(G_1,H_1,L_1)$,
$(G_2,H_2,L_2)$ are homomorphic, if there exists a group
homomorphism $\varphi:G_1\to G_2$ which maps $H_1$ onto $H_2$ and
$L_1$ onto $L_2$. In this case, the map $\varphi$ induces a left-loop
homomorphism ${\varphi}_{|L_1}:L_1\to L_2$. Conversely, every
left-loop homomorphism $\varphi:L_1\to L_2$ induces a homomorphism
$\varphi':G(L_1)\to G(L_2)$ which maps $H(L_1)$ onto $H(L_2)$ and
every group homomorphism $\varphi:G_1\to G_2$ which maps $H_1\leq
G_1$ onto $H_2\leq G_2$ induces a homomorphism
$\varphi':(G_1,H_1,L_1)\to \left(G_2,H_2,\varphi(L_1)\right)$ (see
\cite{Nguyen2003onlooopsandevelopes}). This proves the following Proposition, which can be found in \cite{aschbacher2005onbolloops} and in Theorem 1.11 \cite{lgtlt}, pages 21-22. with a formulation in term of transversals and sections.

\begin{Proposition}\label{casertaepimorphism} 
If $(G, H, L)$ is a left-folder, then $(G(L), H(L),T(L))$ is a faithful left-envelope. Furthermore, the function
$\lambda:\langle L\rangle\to G(L)$ which maps $a\in L$ to
$\lambda_a\in T(L)$ is a left-folder epimorphism from $(\langle
L\rangle, \langle L\rangle\cap H, L)$ onto $(G(L), H(L),
T(L))$.
\end{Proposition}

According to the above Proposition \ref{casertaepimorphism}, one can make no theoretical difference between left-loops and left-envelopes. Nevertheless, since many left-loop categories are characterized by properties of the own enveloping group, the use of envelopes mostly simplify the notation.

\subsubsection{Local loop: Kikkawa's geodesic loop}
Let $M$ be a manifold with an affine connection, let $m\in M$ be a fixed point and let $T_pM$ be the tangent space to $M$ in $p$. The exponential map $\exp_p:N_0\subseteq T_p\to M$ maps a star-shaped neighborhood $N_0$ of the null vector onto a  normal neighborhood $U_m$ of $m$. Assume that in such a neighborhood any two points can be joint by a geodesic. Let $\tau_{p,q}:T_pM\to T_qM$ be the parallel transport with respect to the connection and, for a vector $X_m\in T_m M$, let $X(t):=\tau_{m,\gamma(t)}(X_m)$ be the parallel transport of $X_p$ along the curve $\gamma$. If $\Gamma_{ij}^k$ are the components of the connection, then
$$
(\nabla_{\dot{\gamma(s)}}X(s))_k=\dfrac{d}{dt}X_k(t)+\sum_{ij}\Gamma^k_{ij}(\gamma(s))\dfrac{d}{dt}\gamma_i(t)X_j(t),
$$
thus, if $a$ is any point in $U_m$, then there exists a geodesic arc $\gamma$ which is the solution of the previous differential equation with initial conditions $X(0)=\dot{\gamma}(0)=a$. If the exponential map is defined, then $\gamma(t)=\exp_m tX_a$ for some vector $X_a$. The operation
$$
a\oplus b:=(\exp_b\circ\tau_{m,a}\circ\exp_m^{-1})(a)
$$
gives a point which is not always contained $U_m$, but, if it is the case, the resulting point $a\oplus b$ is unique. The equation $x\oplus a=b$ has always a solution in $U_m$ and it is also unique, while $y\oplus x=b$ may not have one, but, once more, if it has one solution in $U_m$, it is also unique, since the Jacobian $\partial y^i/\partial z_j=\delta_{ij}$ does not vanishes in $m$. 

The right deviations leave $m$ invariant, thus the differential of them are automorphisms of the tangent space $T_m$ in $m$, whence the following two Propositions of Kikkawa.

\begin{Proposition}
If $dR_b^{-1}\tau_{b,ab} dR_b$ does not depend on $b$, then $dH$ is precisely the holonomy group at $p$.  
\end{Proposition}
 
\begin{Proposition}
At the point $p = \pi(e)$ of a reductive homogeneous space $M=
G/H$ with a canonical affine connection, the group of linear transformations of $T_p$ induced by the right inner maps of a differentiable local loop coincides with the local holonomy group defined on a restricted normal neighborhood $U_p$ of $p$.
\end{Proposition}

The previous Proposition is particularly interesting from a geometrical point view. It states that the non-associativity of the loop and the holonomy group of the manifold carry actually the same information. If the holonomy group were trivial, then the loop would be associative, that is, a non-flat surface produces a non-associative loop of symmetries.

\subsection{The sphere}
We give account of some elementary consideration about the real sphere in order to generalize them by means of Lie theory and export this pattern to the $p$-adic case. The real sphere is of course the set 
$$
S^2=\{(x,y,z)\in\mathbb{R}^3:x^2+y^2+z^2=1\}
$$
that is, the locus of all points $(x,y,z)\in\mathbb{R}^3$ which have unitary distance from the origin. It is worth to emphasizes that such identification holds because the Euclidean metric can be defined by a scalar product, unfortunately, a similar argument does not make sense in the $p$-adic case, because there is no relationship between ultrametric and scalar product. \\

Denote by $N=(0,0,1)$ be the north pole of the sphere, the tangent space $T_N(S^2)$ in $N$ to $S^2$ is a $2$ dimensional real vector space, indeed, if $\gamma:I\to S^2\subseteq\mathbb{R}^3(t)$ is a curve in $S^2$ through $\gamma(0)=N$ with $\gamma(t)=(x(t),y(t),z(t))$, then $x(t)^2+y(t)^2+z(t)^2=1$ implies $\gamma(t)\cdot\dot{\gamma}(t)=0$, hence $0=N\cdot\dot{\gamma}(0)=\dot{z}(0)=0$, thus
$$
T_N(S^2)=\langle e_1,e_2\rangle.
$$
If $R_{u,\theta}:S^2\to S^2$ is the rotation which takes $N$ to $P=(x_p,y_p,z_p)\in S^2$ and $\Omega$ is the center of the sphere, then $u=(-y_p,x_p,0)$ and the rotation is $R_{u,\theta}=R_{z,\varphi}R_{y,\theta}$, where $\cos\varphi=x_p/\|\vettore{\Omega P}\|$ and $\cos\theta=z_p$. Thus, the differential $\mathrm{d}_NR_{u,\theta}$ of $R_{u,\theta}$ in $N$ is the linear map between the tangent spaces $T_N(S^2)$ and $T_P(S^2)$
$$
\mathrm{d}_NR_{u,\theta}=(R_{u,\theta})_*:v\in T_N(S^2)\mapsto \mathrm{d}_NR_{u,\theta}(v)\in T_P(S^2).
$$
If $[\gamma]\in T_N(S^2)$ is a tangent vector, then $
\mathrm{d}_NR_{u,\theta}([\gamma])=[(R_{u,\theta}\circ\gamma)(t))]$, that is, 
$$
\mathrm{d}/\mathrm{dt}|_{t=0}(R_{u,\theta}\circ\gamma)(t))=\mathrm{d}/\mathrm{dt}|_{t=0}(R_{u,\theta}\gamma(t))=R_{u,\theta}\dot{\gamma}(0),
$$
whence
$$
\mathrm{d}_NR_{u,\theta}(v^t)=R_{u,\theta}v^t\text{ and } T_P(S^2)=R_{u,\theta}T_N(S^2).
$$
According to the definition given in Part 1., the tangent bundle is therefore
$$
TS^2=\{R_{u,\theta}(N,v):v\in T_N(S^2), R_{u,\theta}\in\mathrm{SO}_3(\mathbb{R}) \},
$$
in particular
$$
[\exp{\alpha A_{\hat{u}}}=I+\alpha A_{\hat{u}}+o(\alpha^2)],\quad {\hat{u}}=x,y,z.
$$ 
The connection on $S^2$ is induce by the metric tensor, the geodesics are the maximal circles and the parallel transport is made by a simple a rotation with the axis through the center. Since the parallel transport along a geodesic $\gamma$ through some point $p$ preserves the angle formed by a  tangent vector at $p$ and the tangent vector to $\gamma$ at $p$, we can express the non-associative operation in terms of a rotation, that is, $a\oplus b=R_{\widehat{u},\theta}(b)$, where $\theta$ is the longitude of $a$ and $\widehat{u}$ is the normal component of the Frenet frame of the geodesic through $N$ and $a$ with origin in $a$. The M\"obius transformations of the Gaussian plane $\mathbb{C}$ can be used to describe a rotation of the Riemannian sphere, thanks to its differential structure. Thus, taking the stereographic projection from the south pole in Cartesian coordinates, $\pi_s:(x,y,z)\in U_s\mapsto{x+iy\over 1+z}$, one has
$$
R_{\widehat{u},\theta}\big((x,y,z)\big)=\pi_s^{-1}\left(
\dfrac{\pi_s\big((x,y,z)\big)+\zeta}{1-\overline{\zeta}\pi_s((x,y,z))}\right)
$$
with $\zeta=\mathrm{tan}(\theta/2)e^{i\varphi}$ and $\widehat{u}=\pi_s^{-1}(ie^{i\varphi})$. Up to the identification of the points of $U_s$ with their complex images under $\pi_s$, the non-associative operation takes the form
$$
a\oplus b:=\dfrac{a+b}{-\overline{a}b+1}
$$
for all $(a,b)$ in some subset $V\subseteq\mathbb{C}^2$. Nevertheless, $P_\infty\oplus a$ remains undefined for $a\not=0,P_\infty$.

\subsubsection{Orthogonal geometry in $Q_p$}
So far we have considered the sphere as a subspace of $\mathbb{R}^3$, we would like to generalize the previous considerations and make them intrinsic. We begin by introducing a representation for the group $\mathrm{SO}_3(K)$ which can be used both in real and $p$-adic case.\\

We recall that for a quadratic $K$-space $(\mathcal{V},q)$, $\mathrm{char} K\ne 2$, the orthogonal group is the group of the automorphisms which leave the scalar product invariant, hence it is isomorphic to the group of all matrices $M\in\mathrm{GL}_n(K)$ such that $MBM^t=B$, $B=(\underline{v}_i\cdot\underline{v}_j)$ for a basis $\{\underline{v}_i\}$, for which one has $\mathrm{det}\,M^2=1$. One calls a \textit{rotation} those matrices which have determinant equal to $1$ (the special orthogonal group $SO(\mathcal{V})$) and \textit{reflexion} the others. For a non isotropic vector $\underline{u}$, i.e. $q(\underline{u})\ne 0$, the map 
$$
\sigma_{\underline{u}}(\underline{v}):=\underline{v}-2\dfrac{\underline{v}\cdot\underline{u}}{Q(\underline{u})}\underline{u}
$$
is a reflexion along $\underline{u}$ or through the hyperplane $\underline{u}^\perp$. In particular, in the basis $\langle\underline{u}\rangle\perp\langle\underline{u}\rangle^\perp$, $\sigma_{\underline{u}}$ is represented by $\mathrm{diag}(-1,1,\ldots,1)$, which has determinant equal to $-1$. As a consequence of Cartan-Dieudonn\'e\footnote{Cfr. f.i. Theorem 6.6 in \cite{grove2002classical} page 48.} each orthogonal transformation in a $n$-dimensional space is the product of $n$ or fewer reflexions, thus, in a $3$-dimensional space, each rotation has a non isotropic axis, being the product of $2$ reflexion and, unless the rotation is trivial, there exists a unique fixed vector $\underline{v}$ which is the generator of two reflexion (hyper)planes, $W_1\cap W_2=\langle\underline{v}\rangle$. If the axis $\underline{v}$ were isotropic, then both planes $W_1,W_2$ would contain an isotropic vector, hence they were hyperbolic and the rotation restricted to $W_1$ and $W_2$ would be the identity\footnote{Cfr. \cite{grove2002classical} exercise 2 page 40.} hence the identity on the whole space $\mathcal{V}=W_1+W_2$ (\footnote{For a general reference about orthogonal geometry\label{casertaorthogonalGeometry} see \cite{diarra1994padicclifford}, \cite{grove2002classical}, \cite{micali1992clifford}.}).\\

Consider the case $\mathcal{V}=Q_p^3$. The field $Q_p$ of $p$-adic numbers can be algebraically extended to each degree, thus, any finite dimensional $p$-adic vector space is an algebraic extension of the base field $Q_p$. Of course,
the geometry of the space changes according to the prime number $p$, more precisely, the geometry depends on which roots lie in $Q_p$. For a finite extension $Q_p\leq K$ of degree $n$, call $e=[|K^*|:|Q_p^*|]$ the \textit{ramification index} and set $f=[k:\mathbb{Z}_p]$, whence $n=ef$. One can prove that there exists exactly one unramified extension of $Q_p$ of degree $f$ which can be obtained by adding to $Q_p$ a primitive $p^f-1$-th root of unity\footnote{Cfr. \cite{koblitzpadic1984} page  67.}. Recall that the field $Q_p$ has a root of $-1$ if and only if $p\equiv 1$ modulo $4$. For our purpose, we will consider the case $p\equiv 3$ modulo $4$, whence $p>2$ and $Q_p(i)$, $i^2=-1$, is an unramified extension over $Q_p$. \\

The quadratic space $(Q_p^3,q)$ with respect to the quadratic form $q(x,y,z)= a_1x^2+a_2y^2+a_3z^2$ with $a_1,a_2,a_3\in Z_p^*$ is isotropic\footnote{Cfr. theorem 3.5.1 page 60 and corollary 3.5.2. page 62 in \cite{kitaoka1999arithmetic}, the more general case only requires the condition $p>2$.}, the form $q(x,y,z)=a_1x^2+a_2y^2+a_3z^2$ is not degenerate since $\mathrm{det}\,\mathrm{diag}(a_1,a_2,a_3)\ne 0$ and $(Q_p^3,q)$ is regular, in particular, $Q_p^3$ has an orthogonal base $\{\underline{v}_i\}$ such that 
$$
Q_p^3=\langle q(\underline{v}_1)\underline{e}_1,q(\underline{v}_2)\underline{e}_2,q(\underline{v}_3)\underline{e}_3\rangle\cong\langle\mathrm{diag}(q(\underline{v}_1),q(\underline{v}_2),q(\underline{v}_3))\rangle.
$$
We choose as the quadratic form the standard scalar product $q(x,y,z)=x^2+y^2+z^2$.\\

\subsubsection{Representation of $\mathrm{SO}_3(K)$}

Given a quadratic $K$-space $(\mathcal{V},\cdot)=(\mathcal{V},q)$, $q(\underline{v}):=\underline{v}\cdot\underline{v}$, one (can) define the \textit{Clifford algebra}\label{casertaclifford} $\mathrm{Cl}(\mathcal{V},q)$ as the tensor algebra $T(\mathcal{V})$ quotiented by the ideal generated by $\underline{v}\otimes\underline{v}-q(\underline{v})1$ or, equivalently, $\underline{v}\otimes\underline{w}+\underline{w}\otimes\underline{v}-\underline{w}\cdot\underline{v}1$. In particular we can identify the Clifford algebra with the free algebra on $\mathcal{V}$ with relations $\underline{v}\underline{w}+\underline{w}\underline{v}=2\underline{v}\cdot\underline{w}1$. \\

Since the Clifford algebra is the universal cover of the algebras such that $\mathcal{V}$ can be projected into preserving the quadratic form, if $\iota:\mathcal{V}\to\mathrm{Cl}(\mathcal{V},q)$ is the immersion, then there exists a natural inclusion $\mathrm{O}(\mathcal{V},q)\rightarrow\mathrm{Aut}(\mathrm{Cl}(\mathcal{V},q))$ such that 
$$
\iota(\sigma_{\underline{u}}(\underline{v}))=-\iota(\underline{u})\iota(\underline{v})\iota(\underline{u})^{-1}
$$
for any given reflexion $\sigma_{\underline{u}}$. Whence, for a rotation $\tau$, which is the product of two reflexions, there exists $P$ in the Clifford algebra such that $\iota(\tau_{\underline{u}}(\underline{v}))=PVP^{-1}$, with $\iota(\underline{v})=V$ and 
$$
\tau_{k\underline{u}}(\underline{v})=\underline{v}-2k\dfrac{\underline{v}\cdot\underline{u}}{k^2 Q(\underline{u})}k\underline{u}=\tau_{\underline{u}}(\underline{v}), \quad 
(kP)V(kP)^{-1}=PVP^{-1},\quad\forall k\in K^*.
$$

According to the classification theorems for Clifford algebras \footnote{Cfr. theorems 4, 4' in \cite{diarra1994padicclifford} page 99. for the $p$-adic case}, if $p\equiv 3$ mod $4$, then 
$\mathrm{Cl}(Q_p^3,Q)\cong\mathrm{M}_2(Q_p(i))$, and we recover the vector space $Q_p^3$ inside $\mathrm{M}_2(Q_p(i))$ thanks to the immersion 
\begin{equation}\label{casertavectorspace}
\iota:(a,b,c)\in Q_p^3\mapsto 
\begin{pmatrix}
c & a+ib\\
a-ib & -c
\end{pmatrix}
\end{equation}
with $\iota(\underline{v})^2=\underline{v}\cdot\underline{v}I=-\mathrm{det}V$ and 
$[\iota(\underline{v}),\iota(\underline{u})]=2\underline{u}\cdot\underline{v}I$, a vector $\underline{v}$ being isotropic if and only if $\mathrm{det}\iota(\underline{v})=0$. 
 
\begin{Proposition}\label{casertarotation-representation}
The group $\mathrm{SO}_3(Q_p(i))$ has a faithful representation in $\mathrm{PGL}_2(Q_p(i))\leq\mathrm{M}_2(Q_p(i))/Q_P^*I$ as the projective group 
$$
\Gamma:=
\left\{
\begin{bmatrix}
\alpha & \beta\\
-\overline{\beta} & \overline{\alpha}
\end{bmatrix}
:\alpha\overline{\alpha}+\beta\overline{\beta}\ne 0,\,\alpha,\beta\in Q_p(i)
\right\}, \quad \overline{a+ib}:=a-ib,\, a,b\in Q_p.
$$
\end{Proposition}
\begin{proof}
See \cite{Caserta2017padicloop}
\end{proof}

\textbf{$p$-adic representation of the sphere}  
For a prime number $p$ and an integer $z$ we denote by $|z|_p$ the (normed) $p$-adic absolute value, by $\mathbb{Z}_p$ the completation of $\mathbb{Z}$ with respect to the absolute value and by $Q_p$ the field of quotients of $\mathbb{Z}_p$. Remember that we are assuming $p\equiv 3$ modulo $4$, therefore, $Q_p(i)$, $i^2=-1$, is the ``unique'' unramified extension over ${Q}_p$ of order two. We will furthermore consider the quadratic space on $Q_p^3$ with respect to the standard scalar product, which is isotropic\footnote{Cfr. theorem 3.5.1 page 60 and corollary 3.5.2. page 62 in \cite{kitaoka1999arithmetic}, the more general case only requires the condition $p>2$.} and not degenerate. We represent the ${Q}_p$-vector space ${Q}_p^3$ in its Clifford algebra as the vector space of Pauli matrices\footnote{The isomorphism between the quaternion and the Pauli matrices sis given by $ai+bj+ck\in\mathbb{H}\mapsto -i(a\sigma_1+b\sigma_2+c\sigma_3)\in\mathrm{Cl}(Q_p^3,\cdot)$}
$$
[a,b,c]:=a\sigma_x+b\sigma_y+c\sigma_z=\begin{pmatrix}
c & a-ib\\
a+ib& -c
\end{pmatrix},\quad a,b,c\in Q_p, 
$$
with $\sigma_x=[1,0,0]$, $\sigma_y=[0,1,0]$ and $\sigma_z=[0,0,1]$. The $2$-sphere $\{(a,b,c)\in Q_p^3: a^2+b^2+c^2=1\}$ is therefore an affine $2$-dimensional manifold of the matrices of order $2$, that is,
\begin{equation}\label{casertasphere}
S^2:=\{a\sigma_x+b\sigma_y+c\sigma_z\in M_2({Q}_p(i)): (a\sigma_x+b\sigma_y+c\sigma_z)^2=I\}.
\end{equation}
Define furthermore the cup $\tilde{S}^2$ around the ``north pole'' $\sigma_3$ by
\begin{equation}\label{casertacup}
\tilde{S}^2=\{P\in S^2:\|P-\sigma_z\|\leq p^{-1}\}
\end{equation}
where $\|\cdot\|$ is the (unique) max/sup-norm defined of the vector space $Q_p^3$. As the following Lemmata will show, this neighborhood of $\sigma_z$ is small enough, and actually the biggest, in order to allow us to parametrize $\tilde{S}^2$ by the trigonometric maps and give a ``polar representations'' which is analogue to the real case.
Of course, the $2$-sphere $S^2$ is also a $2$-dimensional analytic $p$-adic manifold and an analytic atlas is given by the ``stereographic projection'' $\{(\varphi_{\sigma_z}, U_{\sigma_z})$, $(\varphi_{-\sigma_z}, U_{-\sigma_z})\}$, where 
$$
U_{\sigma_z}=S^2\setminus\{\sigma_z\},\,
U_{-\sigma_z}=S^2\setminus\{-\sigma_z\},
$$
both the projections
$$
\varphi_{\sigma_z}:
\begin{pmatrix}
c & z\\
\overline{z} & -c
\end{pmatrix}\in U_{\sigma_z}\mapsto \dfrac{z}{1+c}\in Q_p(i),\,
\varphi_{-\sigma_z}:
\begin{pmatrix}
c & z\\
\overline{z} & -c
\end{pmatrix}\in U_{-\sigma_z}\mapsto \dfrac{z}{1-c}\in Q_p(i)
$$
being analytic, since the conjugation is $Q_p$-analytic as well as the rational functions.

\textbf{Reductive space} In the real case, the sequence $\mathrm{SO}_2(\mathbb{R})\to\mathrm{SO}_3(\mathbb{R})\to S^2$ is a principal bundle, we will make the equivalent $p$-adic ``metrical computation'' via Lie theory. The $\mathbb{Z}_p$-modules $\mathcal{V}\leq\mathrm{Lie}(\Gamma_{\sigma_z})$ and $\mathcal{H}\leq T_{\sigma_z}(S^2)$ defined by
$$
\mathcal{V}=\left\{
\begin{pmatrix}
ia & 0\\
0 & -ia
\end{pmatrix}\Big|\, a\in p\mathbb{Z}_p\right\}
,\quad \mathcal{H}=\left\{
\begin{pmatrix}
0 & \beta\\
-\overline{\beta} & 0
\end{pmatrix}\Big|\, \beta\in p\mathbb{Z}_p(i)\right\}
$$
play an analogue role of the horizontal and vertical distributions in $\sigma_z$, which allow in the real case to define an affine connection and the parallel transport. The loop of Kikkawa type on the $2$-sphere is an affine submanifold of $S^2$ which is orthogonal to $\mathcal{V}$ and has $\mathcal{H}$ as tangent plane in $\sigma_z$. Since the homogeneous space $\mathrm{SO}_3(Q_p)/\mathrm{SO}_2(Q_p)$ is still reductive and $\tilde{S}^2$ is still a transversal, the lift to $\mathrm{SO}_3(Q_p)$ of a curve $\gamma$ on $S^2$ through $\sigma_z$, $\gamma(t)=g_t.\sigma_z$, $g_t\in \mathrm{SO}_3(Q_p)$, $t\in I$, is $\gamma^*(t)=R_{\theta(t),\varphi(t)}.\sigma_z$. Borrowing the definition of parallel displacement induced by a connection, we take the curves $R_{\theta,\varphi(t)}.\sigma_z$ as geodesic curves and the rotation $R_{\theta,\varphi}$ as parallel displacement. Hence, the following.
 
\begin{Proposition}
The set $\exp(\mathcal{V})$ is a subgroup of $\Gamma^0_{\sigma_z}$, the connected component in $\Gamma_{\sigma_z}$ of $0$, and $L:=\exp(\mathcal{H})$ is a local section (a subset of representatives of left cosests) of $(\Gamma^0\cap\exp(\mathcal{V}))/\exp(\mathcal{V})$.
\end{Proposition}
\begin{proof}
See \cite{Caserta2017padicloop}
\end{proof}

In particular, the group $\exp(\mathcal{V})$ is the group of elements 
$$
\begin{pmatrix}
\cos 2\theta & \exp(i\varphi)\sin 2\theta\\
\exp(-i\varphi)\sin 2\theta & -\cos 2\theta
\end{pmatrix}
$$ 
with $\varphi\in\mathbb{Z}_p$ and $\theta=\sqrt{z\overline{z}}$ where $z=\theta i\exp(i\varphi)\in p\mathbb{Z}_p(i)$ and it is also isomorphic to the projective group of elements
$$
\begin{bmatrix}
1 & \exp(i\varphi)\tan 2\theta\\
-\exp(-i\varphi)\tan 2\theta & 1
\end{bmatrix}
$$

\begin{Proposition}
The projection
\begin{equation}\label{casertastereographic}
\psi:
\begin{pmatrix}
c & a+ib\\
a-ib & -c
\end{pmatrix}\in \tilde{S}^2\mapsto \dfrac{a+ib}{1+c}\in Q_p(i),\,
\end{equation}
provides an isomorphism (diffeomorphism) between $D=\{z\in Q_p(i): |z|_p<1\}$ and $\tilde{S}^2$. In particular
\begin{equation}\label{casertasterographic2}
\psi\left(\begin{pmatrix}
\cos 2\theta & \exp(i\varphi)\sin 2\theta\\
\exp(-i\varphi)\sin 2\theta & -\cos 2\theta
\end{pmatrix}\right)=\exp(i\varphi)\tan\theta
\end{equation}
with
\begin{equation}\label{casertastima} 
|\exp(i\varphi)\tan\theta|_p=|\theta|_p<1.
\end{equation}
\end{Proposition}

\begin{proof}
See \cite{Caserta2017padicloop}
\end{proof}

\subsubsection{Loop of Kikkawa type}

Following Kikkawa, for $A,B\in\tilde{S}^2$, $A=R_{\varphi,\theta}.\sigma_z$, we define the operation 
$$
A\oplus B:=\left(\exp_B\circ\tau_{N,B}\circ\exp_N^{-1}\right)(A)
$$ 
where $\exp_P$ is the exponential map which maps a tangent vector $X$ in $P$ onto $\gamma_P(1)$, $\gamma$ being the integral curve in $P$ of $X$ and $\tau_{N,B}$ being the parallel displacement along the geodesic which joins $N,B$. We can entirely project the cup $\tilde{S}^2$ diffeomorphically onto the disk $D=\{\xi\in Q_p(i):|\xi|_p<1\}$, thus, as in the real case, the operation becomes
$$
\xi_1\oplus\xi_2=\dfrac{\xi_1+\xi_2}{1-\overline{\xi}_1\xi_2}, \quad \xi_1,\xi_2\in D,
$$
and simple computation show that $(D,0,\oplus)$ is an $A_l$-loop (and not only a local loop) with the left inverse property. Indeed, if we represent the points projectively, we have
$$
\lambda_{\xi_1}(\xi_2)=
\begin{bmatrix}
\xi_1\\
1
\end{bmatrix}
\oplus
\begin{bmatrix}
\xi_2\\
1
\end{bmatrix}
=
\begin{bmatrix}
1 & \xi_1\\
-\bar{\xi}_1 & 1
\end{bmatrix}
\begin{bmatrix}
\xi_2\\
1
\end{bmatrix}
$$
with deviations
$$
\delta_{\xi_1,\,\xi_2}=\lambda_{\xi_1\oplus\xi_2}^{-1}\lambda_{\xi_1}\lambda_{\xi_2}=
\begin{bmatrix}
1-\xi_1\bar{\xi}_2 & 0 \\
0 & 1-\bar{\xi}_1\xi_2
\end{bmatrix}
$$
which act on $D$ by
$$
\delta_{\xi_1,\, \xi_2}(\xi)=\dfrac{1-\xi_1\bar{\xi}_2}{ 1-\bar{\xi}_1\xi_2}\xi,\quad \overline{\dfrac{1-\xi_1\bar{\xi}_2}{ 1-\bar{\xi}_1\xi_2}}=\left(\dfrac{1-\xi_1\bar{\xi}_2}{ 1-\bar{\xi}_1\xi_2}\right)^{-1},\quad \left|\dfrac{1-\xi_1\bar{\xi}_2}{ 1-\bar{\xi}_1\xi_2}\right|_p=1.
$$

\section{Appendix}

\appendix

\subsection{$p$-adic algebraic and transcendental functions}

Let $p$ a prime number, for all $z\in\mathbb{Z}^*$, let $v_p(z)$ be the greatest power of $p$ which divides $z$. Define the \textit{(normed) $p$-adic absolute value} of $z$ by setting $|z|:=p^{-v_p(z)}$. The completation $Z_p$ of $\mathbb{Z}$ with respect to the previous absolute value is the \textit{ring of $p$-adic integers} and the field of quotients $Q_p$ of $Z_p$ is the \textit{$p$-adic field}, which is a local field with characteristic $0$ and a totally disconnected metric space with respect to the distance induced by the absolute value.
We define the \textit{square root} of a $p$-adic number by means of the $p$-adic binomial series, which is defined, for $\alpha\in Z_p$, by
$$
\sum_{n=0}^\infty\binom{\alpha}{n}x^n
$$
and converges for $|x|<1$ (cfr. \cite{gouvea1997p} page 123-124). The square roots of $-1$ are denoted as usual by $\pm i$ where $-1$ is the opposite of $1$.

We give now account of the classical transcendental $p$-adic maps, for references see \cite{koblitzpadic1984} page 81, \cite{robert2000course} page 241, \cite{schikhof2007ultrametric} page  128., \cite{gouvea1997p} page  112. The definition of these maps are via power series. Since
$$
v_p(n!)=\sum_{i=0}^\infty \left\lfloor\dfrac{n}{p^i}\right\rfloor<\dfrac{n}{p-1}
$$
in particular $|n!|>p^{-n/(p-1)}$ (cfr. \cite{gouvea1997p} page  114), by setting $r=p^{-1/(p-1)}$, $1/p<r$ and  $D=B_r(0):=\{x\in Z_p:|x|<r\}$ as convergence ray and disk respectively the following propositions \ref{casertaexp},\ref{casertaexp-prop}, \ref{casertaexp-sin-cos} and \ref{casertaarctg} hold\footnote{Cfr. also \cite{gouvea1997p} page 112-120. For the proofs of \ref{casertaexp},\ref{casertaexp-prop}, \ref{casertaexp-sin-cos} and \ref{casertaarctg} see  page  70-71, page  128, page  135-136 and page  137. in \cite{schikhof2007ultrametric} respectively.}.

\begin{Proposition}\label{casertaexp}
The series
$$
\exp_p x :=\sum_{n=0}^\infty \dfrac{x^n}{n!},\quad \log_p(1+ x) :=\sum_{n=1}^\infty (-1)^{n+1}\dfrac{x^n}{n} 
$$
both converge in $D$ and for $|x|<1$ respectively. Moreover
$$
\sin x=\sum_{n=0}^\infty (-1)^n\dfrac{x^{2n+1}}{(2n+1)!},\quad
\cos x=\sum_{n=0}^\infty (-1)^n\dfrac{x^{2n}}{(2n)!}
$$
also converge in $D$.
\end{Proposition}

\begin{Proposition}\label{casertaexp-prop}
The following hold
\begin{itemize}
\item $\exp(x+y)=\exp x\, \exp y$, $x,y\in D$
\item $(\exp x)'=\exp x$
\item $\log(xy)=\log x+\log y$, $|x|,|y|<1$
\item $\log\exp x=x$, $\exp\log y=y$, $x\in E$, $y\in 1+E$ 
\end{itemize}
\end{Proposition}

\begin{Proposition}\label{casertaexp-sin-cos}
For all $x,y\in D$ the following hold
\begin{itemize}
\item $\exp ix=\cos x +i\sin y$
\item $(\sin x)^2+(\cos y)^2=1$
\item $\sin(x+y)=\sin x\cos y+\cos x\sin y$
\item $\cos(x+y)=\cos x\cos y-\sin x\sin y$
\item $|\sin x|=|x|$, $|\cos x|=|\exp ix|=1$
\item $|\sin x-x|<|x|$, $|\cos x-1|<|x|$ ($x\ne 0$)
\item $|\sin x-\sin y|=|x-y|$, $|\cos x-\cos y|\leq |x-y|$
\end{itemize}
Furthermore the cosinus has no zeros and the unique zero of the sinus is $0$ (\footnote{Cfr. exercise 25.1 in  \cite{schikhof2007ultrametric} page 72.}).
\end{Proposition}

\begin{Proposition}[van Hamme]\label{casertaarctg}
For all $x\in C_p$, $x\ne \pm i$
$$
\arctan x:=\dfrac{1}{2i}\log\left(\dfrac{1+ix}{1-ix}\right)
$$
and for all $x\in C_p$, $|x|<1$
$$
\arcsin x:=\dfrac{1}{i}\log\left(ix+\sqrt{1-x^2}\right).
$$ 
\end{Proposition}

\nocite{*}
\bibliographystyle{amsplain}

\end{document}